\documentclass[reqno]{amsart}

\usepackage[T1]{fontenc}
\usepackage{enumerate, amsmath, amsfonts, amssymb, amsthm, mathrsfs}
\usepackage{bbm}

\usepackage[lmargin=3.5cm,rmargin=3.5cm,top=2.9cm,bottom=2.9cm]{geometry}

\usepackage[all]{xy}
\usepackage{tikz}
\usetikzlibrary{arrows, decorations, shapes}
\usepackage{tikz-cd}
\usepackage{colonequals}
\usetikzlibrary{arrows}
\usepackage[colorlinks, citecolor=blue]{hyperref}
\usepackage[noabbrev,capitalise]{cleveref}
\usepackage{multirow}
\usepackage{ulem}
\usepackage{cancel}
\usepackage{caption} 
\usepackage{enumitem}
\setlist[enumerate]{format=\normalfont}

\newtheorem{theorem}{Theorem}[section]
\newtheorem{corollary}[theorem]{Corollary}
\newtheorem{proposition}[theorem]{Proposition}
\newtheorem{lemma}[theorem]{Lemma}
\newtheorem*{theorem*}{Theorem}

\theoremstyle{definition}
\newtheorem{definition}[theorem]{Definition}
\newtheorem{example}[theorem]{Example}
\newtheorem{remark}[theorem]{Remark}

\newtheorem{setup}[theorem]{Setup}
\newtheorem{algorithm}[theorem]{Algorithm}
\newtheorem{construction}[theorem]{Construction}

\newcommand{\ac}{\mathcal{A}}
\newcommand{\cc}{\mathcal{C}}
\newcommand{\fc}{\mathcal{F}}
\newcommand{\mc}{\mathcal{M}}
\newcommand{\tc}{\mathcal{T}}

\newcommand{\ucc}{\mathfrak{U}}
\newcommand{\uc}{\mathcal{U}}
\newcommand{\vc}{\mathcal{V}}
\newcommand{\xc}{\mathcal{X}}
\newcommand{\kk}{\mathbbm{k}}
\newcommand{\os}{\operatorname{os}}

\newcommand{\modA}{\operatorname{mod}A}
\renewcommand{\mod}{\operatorname{mod}}

\newcommand{\Ext}{\operatorname{Ext}}
\newcommand{\End}{\operatorname{End}}

\newcommand{\Coker}{\operatorname{Coker}}
\newcommand{\Hom}{\operatorname{Hom}}
\newcommand{\Rad}{\operatorname{Rad}}
\newcommand{\add}{\operatorname{add}}

\newcommand{\dq}{\operatorname{dq}}
\newcommand{\tor}{\operatorname{tors}}
\newcommand{\dtor}{d\text{-}\!\operatorname{tors}}
\newenvironment{psmallmatrix}
{\left(\begin{smallmatrix}}
	{\end{smallmatrix}\right)}

\newcommand{\Rep}[1]{%
	{%
		\small%
		\begin{matrix}%
			#1%
		\end{matrix}%
	}%
}

\newcommand{\rep}[1]{%
	{%
		\tiny%
		\begin{matrix}%
			#1%
		\end{matrix}%
	}%
}

\usepackage{setspace}
\newcommand{\marginparstretch}{0.6}
\let\oldmarginpar\marginpar
\renewcommand\marginpar[1]{\-\oldmarginpar[\framebox{\setstretch{\marginparstretch}\begin{minipage}{\marginparwidth}{\raggedleft\tiny #1}\end{minipage}}]{\framebox{\setstretch{\marginparstretch}\begin{minipage}{\marginparwidth}{\raggedright\tiny #1}\end{minipage}}}}

\title{A characterisation of higher torsion classes}

\author[August]{Jenny August}
\address{Department of Mathematics\\ Aarhus Universitet\\ Ny Munkegade 118\\ DK-8000 Aarhus C\\ Denmark}
\email{jennyaugust@math.au.dk}
\author[Haugland]{Johanne Haugland}
\address{Department of Mathematical Sciences\\ 
	NTNU\\ 
	NO-7491 Trondheim\\ 
	Norway}
\email{johanne.haugland@ntnu.no}
\author[Jacobsen]{Karin M.\ Jacobsen}
\address{Department of Mathematics\\ Aarhus Universitet\\ Ny Munkegade 118\\ DK-8000 Aarhus C\\ Denmark}
\email{karin.jacobsen@ntnu.no}
\author[Kvamme]{Sondre Kvamme}
\address{Department of Mathematical Sciences\\ 
	NTNU\\ 
	NO-7491 Trondheim\\ 
	Norway}
\email{sondre.kvamme@ntnu.no}
\author[Palu]{Yann Palu}
\address
{LAMFA, Universit\'e de Picardie Jules Verne, 33, rue Saint-Leu 80039 Amiens, France}
\email{yann.palu@u-picardie.fr}
\author[Treffinger]{Hipolito Treffinger}
\address{IMAS-CONICET, Pabellon I – Ciudad Universitaria, Buenos Aires, 1428, Argentina}
\email{htreffinger@dm.uba.ar}

\usepackage{url}
\begin{document}
	
	
	\maketitle
	
	\begin{abstract}
		Let $\ac$ be an abelian length category containing a $d$-cluster tilting subcategory $\mc$. We prove that a subcategory of $\mc$ is a $d$-torsion class if and only if it is closed under $d$-extensions and $d$-quotients. This generalises an important result for classical torsion classes. As an application, we prove that the $d$-torsion classes in $\mc$ form a complete lattice. Moreover, we use the characterisation to classify the $d$-torsion classes associated to higher Auslander algebras of type $\mathbb{A}$, and give an algorithm to compute them explicitly. The classification is furthermore extended to the setup of higher Nakayama algebras.
	\end{abstract}

	\setcounter{tocdepth}{1}
	\tableofcontents
	
	\section{Introduction}
	
	The notion of torsion pairs was introduced for arbitrary abelian categories in \cite{Dickson} to generalise the properties of the class of torsion groups in the category of abelian groups. Since then, torsion theories, and the related notion of $t$-structures \cite{BeilinsonBernsteinDeligne} for triangulated categories, have become ubiquitous in representation theory, homological algebra, and algebraic geometry. Within these areas, torsion theories and $t$-structures play a key role in topics such as (perverse) sheaves \cite{BeilinsonBernsteinDeligne}, tilting theory and its generalisations \cite{BrennerButler, AdachiIyamaReiten, AngeleriMarksVitoria, Mutation and torsion pairs}, and stability conditions \cite{Bridgeland}.
	
	Meanwhile, higher homological algebra has become an increasingly active field of research since its introduction in  \cite{Iyama2007a, Iyama2007b, Iyama2011}. It has found applications in algebraic K-theory \cite{DyckerhoffJassoWalde}, wrapped Floer theory in symplectic geometry \cite{DyckerhoffJassoLekili}, and in algebraic geometry where it was an important ingredient in the proof of the Donovan-Wemyss conjecture \cite{JassoKellerMuro}.
	Originally motivated by Auslander--Reiten theory, cluster tilting theory and the classical Auslander correspondence, one studies categories where the role of short exact sequences (or distinguished triangles) is taken by longer sequences. Examples include $d$-abelian and $d$-exact categories \cite{Jasso}, $(d+2)$-angulated categories \cite{GeissKellerOppermann}, and $d$-exangulated categories \cite{HerschendLiuNakaoka}. Here, the positive integer $d$ controls the length of the important sequences, with $d=1$ coinciding with the classical cases.
	
	By work of \cite{Jasso, Kvamme, Ebrahimi}, studying $d$-abelian categories is equivalent to studying so called $d$-cluster tilting subcategories of abelian categories. Our setup will be the latter, where we assume the ambient abelian category to be of finite length (see \cref{subsec:d-abelian categories} for details).
	Many important classical concepts in representation theory generalise to this setting. In this paper, we focus on the higher analogue of torsion classes, namely $d$-torsion classes, introduced in \cite{Jorgensen} and further studied in \cite{AJST}.
	
	A fundamental result in the study of torsion classes states that a subcategory of an abelian length category is a torsion class if and only if it is closed under extensions and quotients \cite{Dickson}. This result is of crucial importance, as it both allows for the detection of torsion classes, and moreover gives properties which play a key role in many proofs related to these objects. A higher-dimensional version of this classical characterisation would hence be a substantial advancement. The main result of this paper gives such a characterisation of $d$-torsion classes in terms of closure under $d$-extensions and $d$-quotients (see \cref{def: d-ext closed,def: dq closure}). 
	Note that throughout this paper we assume subcategories to be closed under finite direct sums and summands, see the subsection on conventions and notation below.
	
	\begin{theorem}[\cref{cor: characterisation}] \label{intro: characterisation}
		Let $\mc$ be a $d$-cluster tilting subcategory of an abelian length category $\ac$. A subcategory of $\mc$ is a $d$-torsion class if and only if it is closed under \mbox{$d$-extensions} and $d$-quotients.
	\end{theorem}
	
	Note that while a different characterisation of $d$-torsion classes using the bounded derived category of $\ac$ has been given in the special case where $\ac$ is the module category of a $d$-representation finite $d$-hereditary algebra \cite{Jorgensen}, the characterisation in \cref{intro: characterisation} works generally, does not require use of derived categories, and is closer to the classical result for torsion classes.
	
	The proof of \cref{intro: characterisation} makes significant use of the main result in \cite{AJST}, which relates $d$-torsion classes in $\mc$ to torsion classes in $\ac$. This allows us to apply results about torsion classes also in the higher setup. As another key ingredient in the proof of \cref{intro: characterisation}, we give a higher generalisation of the classical factorisation of a morphism in an abelian category as the composition of an epimorphism followed by a monomorphism, see \cref{Lemma: Coimage factorization}. This result is of independent interest, and we expect it to play a role in providing answers to other  questions in higher homological algebra. 
	
	With \cref{intro: characterisation} in hand, we are able to generalise other well-known results about torsion classes to the higher setting. For example, an immediate consequence of combining \cref{intro: characterisation} with a result in \cite{Klapproth}, is that every $d$-torsion class carries the structure of a $d$-exact category, see \cref{d-torsion is d-exact}.
	
	Other important results concerning classical torsion classes include the study of their poset structure. The set $\tor(\ac)$ of torsion classes in $\ac$ has a natural partial order given by inclusion, and this poset is actually a complete lattice, with meet given by intersection, see e.g.\ \cite[Proposition 2.3]{IyamaReitenThomasTodorov}. \cref{intro: characterisation} allows us to give the following generalisation of this result.
	
	\begin{theorem}[\cref{theorem:lattice}]\label{thm: intro lattice}
		Let $\mc$ be a $d$-cluster tilting subcategory of an abelian length category $\ac$.
		Then the set $\dtor(\mc)$ of all $d$-torsion classes in $\mc$ is a complete lattice with meet given by intersection.
	\end{theorem}
	
	In the classical setting, lattice-theoretic properties of $\tor(\ac)$ form an area of active research \cite{AP,GM,Jasso2}, which is intimately related to representation theory \cite{BCZ,DIRRT}. We refer to \cite{Thomas} for an introductory survey on the topic.  \cref{thm: intro lattice} opens up a new avenue of research in higher homological algebra through the study of the lattice of higher torsion classes.
	
	The classical characterisation of torsion classes as those subcategories that are closed under extensions and quotients, is often used to determine if a given subcategory is a torsion class, or to compute the smallest torsion class containing that given subcategory. Crucial to the success of this approach, is that we often have a good understanding of the middle terms in extensions. To use \cref{intro: characterisation} in an analogous way to determine $d$-torsion classes, we therefore need an understanding of middle terms in $d$-extensions, which is notably more complicated. The following result simplifies this problem significantly. 
	
	\begin{theorem}[\cref{prop: extension closure indec}] \label{intro: indec}
		Let $\mc$ be a $d$-cluster tilting subcategory of an abelian length category $\ac$. Suppose $\uc \subseteq \mc$ is closed under  under $d$-extensions with indecomposable end terms and all $d$-quotients. Then $\uc$ is closed under all $d$-extensions.
	\end{theorem}
	
	This theorem is of independent interest, as extension closure is a useful concept across many areas of representation theory. In this paper, we focus on using \cref{intro: indec} to classify $d$-torsion classes in concrete examples. 
	
	We apply our results to higher Auslander algebras of type $\mathbb{A}$ \cite{Iyama2011} and higher Nakayama algebras of type $\mathbb{A}$ and $\mathbb{A}_\infty^\infty$ \cite{Higher Nakayama}. The module category of each such algebra contains a $d$-cluster tilting subcategory, and their combinatorial descriptions due to \cite{OppermannThomas} and \cite{Higher Nakayama} make them an ideal testing ground for new results in higher homological algebra. Higher Auslander algebras are particularly important, as their derived categories are equivalent to certain partially wrapped Fukaya categories  \cite{DyckerhoffJassoLekili}. We use \cref{intro: characterisation,intro: indec} to give a combinatorial description of all $d$-torsion classes associated to these algebras, where the classification results for the three families are given in \cref{thm:higherAuslanderalg,thm:higherNakayamaAlg,thm:higherNakayamaAlgInfinite}, respectively.
	Moreover, 
	we implement our results in algorithms that compute and count all $d$-torsion classes, see \cref{tab:Auslander torsion classes} and \cref{rem:results-higher-Nakayama}.
	
	We expect that the results we present in this article will provide tools for any further study of $d$-torsion classes and will be of importance in building bridges between $d$-torsion classes and other subjects in representation theory and beyond. This is demonstrated in a forthcoming paper \cite{AHJKPT2}, where \cref{intro: characterisation} is applied to establish a connection between functorially finite $d$-torsion classes, $\tau_d$-tilting theory and $(d+1)$-term silting objects.
	
	\subsection*{Structure of the paper}
	
	In \cref{sec:background} we give an overview of the definitions and background  used in the rest of the paper. In \cref{sec: general results} we prove \cref{intro: characterisation,intro: indec}, while \cref{thm: intro lattice} is shown in \cref{sec:lattice}. \cref{sec: higher AA} is dedicated to the study of $d$-torsion classes associated to higher Auslander algebras. Finally, we extend our view to higher Nakayama algebras in \cref{sec:Higher NA}.
	
	\subsection*{Conventions and notation}
	
	Throughout this paper, let $d$ denote a positive integer and $\ac$ an essentially small abelian category. We always assume $\ac$ to be a finite length category, which implies that the Krull--Remak--Schmidt property is satisfied, see \cite[Lemma 5.1 and Theorem 5.5]{Krause}. 
	
	We let $\kk$ be a field. Given a finite-dimensional $\kk$-algebra $A$, the notation $\modA$ is used for the category of finitely presented right $A$-modules. 
	
	Arrows in a quiver are composed from left to right, meaning that we write $ab$ for the path starting in the source of $a$ and ending in the target of $b$.
	
	All subcategories are assumed to be full and closed under isomorphisms and finite direct sums. They are also assumed to be closed under direct summands. For a collection of objects $\xc$ in an additive category $\cc$, we denote by $\add(\xc)$ the smallest subcategory of $\cc$ which contains $\xc$ and is closed under finite direct sums and direct summands.
	
	\section{Background and preliminaries}\label{sec:background}
	
	In this section we provide an overview of definitions and results which give the foundation for the rest of the paper. Before we start discussing notions from higher homological algebra, we recall some terminology related to subcategories and approximations.
	
	Let $\xc$ be a subcategory of the abelian category $\ac$. We say that $\xc$ is \textit{generating} if any object in $\ac$ is a quotient of an object in $\xc$; that is, for every $Y \in \ac$, there exists an exact sequence $X \to Y \to 0$ with $X \in \xc$. Dually, we can define the notion of \textit{cogenerating}, and we call a subcategory \textit{generating-cogenerating} if it is both generating and cogenerating.
	
	Given an object $Y \in \ac$, a morphism $f \colon Y \rightarrow X$ with $X\in\xc$ is a \textit{left $\xc$-approximation} of $Y$ if any morphism $Y \rightarrow X'$ with $X'\in\xc$ factors through $f$. The subcategory $\xc$ is called \textit{covariantly finite} if every object in $\ac$ admits a left $\xc$-approximation. The notions of \mbox{\textit{right $\xc$-approximations}} and \textit{contravariantly finite} subcategories are defined dually, and a subcategory is \textit{functorially finite} if it is both covariantly and contravariantly finite. 
	
	Recall that a morphism $f \colon X \rightarrow Y$ is called \textit{left minimal} if any endomorphism $g$ of $Y$ satisfying $g \circ f = f$ is an isomorphism. A \textit{minimal left $\xc$-approximation} is a left $\xc$-approximation which is also left minimal. \textit{Right minimal morphisms} and \textit{minimal right} $\xc$\textit{-approximations} are defined dually. Since $\ac$ is Krull--Schmidt, an object has a left (right) $\xc$-approximation if and only if it has a minimal left (right) $\xc$-approximation (\cite[Corollary 1.4]{KrauseSaorin}).
	
	\subsection{\texorpdfstring{$d$}{d}-cluster tilting subcategories and \texorpdfstring{$d$}{d}-abelian categories}
	\label{subsec:d-abelian categories}
	
	The theory of higher homological algebra originated in \cite{Iyama2007a,Iyama2007b} with the study of \textit{$d$-cluster tilting subcategories}. The definition is given below.
	
	\begin{definition}
		A functorially finite generating-cogenerating subcategory $\mc$ of the abelian category $\ac$ is \textit{$d$-cluster tilting} if
		\begin{align*}
			\mc  &= \{X \in \ac \mid \Ext_\ac^i(X,M)=0 \text{ for  $M \in \mc$ and  $i = 1,\ldots, d-1$}\}\\
			&= \{Y \in \ac \mid \Ext_\ac^i(M,Y)=0 \text{ for $M \in \mc$ and $i = 1,\ldots, d-1$}\}.
		\end{align*}
	\end{definition}
	
	To formalise the homological structure of $d$-cluster tilting subcategories, Jasso introduced \textit{$d$-abelian categories}, where the case $d=1$ recovers the classical notion of abelian categories \cite{Jasso}. To give a precise definition, we first recall some terminology. 
	
	Let $\mc$ be an additive category and recall that a \textit{weak cokernel} of a morphism $f \colon X \to Y$ in $\mc$ is a morphism $g\colon Y \to Z$ in $\mc$ for which the induced sequence
	\[
	\Hom_\mc(Z,M) \to \Hom_\mc(Y,M) \to \Hom_\mc(X,M)
	\]
	is exact for any $M\in\mc$. This is equivalent to saying that $g \circ f = 0$ and that for any $g'\colon Y \to M$ with $g' \circ f = 0$, there exists a (not necessarily unique) morphism $h \colon Z \to M$ such that $h \circ g = g'$. We call a morphism $g$ a weak cokernel if there exists a morphism $f$ such that $g$ is a weak cokernel of $f$.
	
	A \textit{$d$-cokernel} of a morphism $f_0\colon X_0 \to X_1$ in $\mc$ is given by a sequence of morphisms 
	\begin{align*}
		X_1 \xrightarrow{f_1} X_2 \xrightarrow{f_2} \dots \xrightarrow{f_{d-1}} X_{d} \xrightarrow{f_d} X_{d+1} \to 0
	\end{align*}
	in $\mc$ such that for every $M$ in $\mc$, the sequence
	\[
	0 \rightarrow \Hom_\mc(X_{d+1}, M) \rightarrow \Hom_\mc(X_d, M) \rightarrow \cdots \rightarrow \Hom_\mc(X_1, M) \rightarrow \Hom_\mc(X_0, M) 
	\]
	of abelian groups is exact. Such a $d$-cokernel is sometimes simply denoted by $(f_1, \dots, f_{d})$, and $(f_1, \dots, f_{d})$ is a $d$-cokernel of $f_0$ if and only if $f_i$ is a weak cokernel of $f_{i-1}$ for $i= 1,\ldots, d-1$ and $f_d$ is the cokernel of $f_{d-1}$. The notion of a \textit{$d$-kernel} in $\mc$ is defined dually. A sequence
	\begin{align}
		0 \to X_0 \xrightarrow{f_0} X_1 \xrightarrow{f_1} \dots \xrightarrow{f_{d-1}} X_{d} \xrightarrow{f_d} X_{d+1} \to 0 \label{eq: d ext}
	\end{align}
	in $\mc$ is called a \textit{$d$-exact sequence} or a \textit{$d$-extension} if $(f_1, \dots, f_{d})$ is a $d$-cokernel of $f_0$ and $(f_0, \dots, f_{d-1})$ is a $d$-kernel of $f_d$. 
	
	\begin{definition}\cite[Definition 3.1]{Jasso}
		An additive category $\mc$ is \textit{$d$-abelian} if it is idempotent complete, every morphism admits a $d$-kernel and a $d$-cokernel, and every monomorphism $f_0$ (resp.\ epimorphism $f_d$) fits into a $d$-exact sequence of the form \eqref{eq: d ext}.
	\end{definition}
	
	A $d$-exact sequence  of the form \eqref{eq: d ext} is said to be \textit{equivalent} to a $d$-exact sequence 
	\[
	0 \to X_0 \xrightarrow{f'_0} X'_1 \xrightarrow{f'_1} \dots \xrightarrow{f'_{d-1}} X'_{d} \xrightarrow{f'_d} X_{d+1} \to 0
	\]
	if there exists a commutative diagram
	\[
	\begin{tikzcd}[column sep=22, row sep=30]
		0 \arrow[r] & X_0 \arrow[r,"f_0"] \arrow[d, equal] & X_1 \arrow[r, "f_1"] \arrow[d] & \cdots \arrow[r, "f_{d-1}"] & X_d \arrow[r, "f_d"] \arrow[d] & X_{d+1} \arrow[d, equal]  \arrow[r]  & 0\\
		0 \arrow[r] & X_0 \arrow[r, "f'_0"] & X'_1 \arrow[r, "f'_1"] & \cdots \arrow[r, "f'_{d-1}"] & X'_{d} \arrow[r, "f'_{d}"] & X_{d+1} \arrow[r]  & 0.
	\end{tikzcd}
	\]
	Note that this defines an equivalence relation on the class of $d$-exact sequences whenever the category $\mc$ is $d$-abelian \cite[Proposition 4.10]{Jasso}.
	
	When $\mc\subseteq\ac$ is a $d$-cluster tilting subcategory, $d$-exact sequences coincide precisely with exact sequences of the form \eqref{eq: d ext} where all terms are in $\mc$. Moreover, any exact sequence in $\ac$ of the form \eqref{eq: d ext} with end terms in $\mc$ is equivalent to one where all terms are in $\mc$ \cite[A.1]{Iyama2007a}.
	
	In a $d$-abelian category, we also find higher analogues of the classical notions of pushouts and pullbacks. For more details on the construction of $d$-pushouts and $d$-pullbacks, we refer the reader to \cite[Section 2.3]{Jasso}.
	
	Jasso proved the following theorem, which shows that $d$-abelian categories capture the homological structure of $d$-cluster tilting subcategories.
	
	\begin{theorem}\cite[Theorem 3.16]{Jasso} \label{thm: d-ct are d-abelian}
		Let $\mc$ be a $d$-cluster tilting subcategory of $\ac$. Then $\mc$ is $d$-abelian.
	\end{theorem}
	
	It has recently been shown that the converse of \cref{thm: d-ct are d-abelian} also holds, i.e.\ that any \mbox{$d$-abelian} category is equivalent to a $d$-cluster tilting subcategory of an abelian category \cite{Ebrahimi,Kvamme}. Therefore, all $d$-abelian categories may be treated as $d$-cluster tilting subcategories.
	
	We end this subsection by introducing a running example that will help us illustrate many of the results in this paper.
	
	\begin{example}\label{ex:running}
		Consider the quiver \(1 \xrightarrow{\alpha} 2 \xrightarrow{\beta} 3\). Let $A$ denote the path algebra of this quiver modulo the ideal generated by the relation $\alpha\beta$. \cref{fig:Ar-quiverA} shows the Auslander--Reiten quiver of $\modA$, where the dashed arrows indicate the Auslander--Reiten translation.
		The subcategory
		\[
		\mc = \add\left\{ \rep{3} \oplus \rep{2\\3} \oplus \rep{1\\2} \oplus \rep{1}	 \right\} 
		\]
		is $2$-cluster tilting in $\modA$, and $\mc$ is hence an example of a $2$-abelian category.
		The indecomposable objects of $\modA$ that generate $\mc$ are marked in \cref{fig:Ar-quiverA}.
		
		\begin{figure}[ht]
			\centering
			\begin{tikzpicture}[line cap=round,line join=round ,x=2.0cm,y=1.8cm, scale = 1]
				\clip(-2.3,0.8) rectangle (2.1,2.5);
				\draw [->] (-1.8,1.2) -- (-1.2,1.8);
				\draw [->] (0.2,1.2) -- (0.8,1.8);
				\draw [<-, dashed] (-1.8,1.0) -- (-0.2,1.0);
				\draw [<-, dashed] (0.2,1.0) -- (1.8,1.0);
				\draw [->] (-0.8,1.8) -- (-0.2,1.2);
				\draw [->] (1.2,1.8) -- (1.8,1.2);
				
				\begin{scriptsize}
					\draw (-2,1) node[draw] {$\Rep{3}$};
					\draw (0,1) node {$\Rep{2}$};
					\draw (2,1) node[draw]  {$\Rep{1}$};
					\draw (-1,2) node[draw]  {$\Rep{2\\3}$};
					\draw (1,2) node[draw]  {$\Rep{1\\2}$};
				\end{scriptsize}
			\end{tikzpicture}
			\caption{The Auslander--Reiten quiver of the module category considered in \cref{ex:running}, with the generators of the $2$-cluster tilting subcategory $\mc$ marked.}
			\label{fig:Ar-quiverA}
		\end{figure}
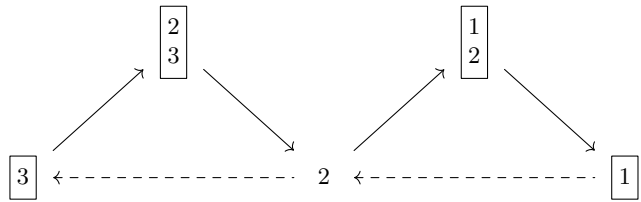
		
	\end{example}
	
	\subsection{Minimality} \label{section: minimal}
	The $d$-kernels and $d$-cokernels in a $d$-abelian category are unique only up to homotopy. Many of our proofs require a stronger sense of uniqueness, which is why we recall the concept of minimality in this section.
	
	The Jacobson radical of the abelian category $\ac$ consists of the morphisms
	\begin{align*}
		\scalebox{0.95}{$\Rad_\ac(X,Y)=\{ f \in \Hom_\ac(X,Y) \mid 1_X - g \circ f \ \text{is invertible for all $g \in \Hom_\ac(Y,X)$}\}.$}
	\end{align*}
	For more details and basic properties, see \cite[A.3]{bluebook1}.
	
	\begin{definition}\cite[Definition 2.5]{HerschendJorgensen}\label{def: minimal}
		Let $\mc$ be a $d$-cluster tilting subcategory of $\ac$.
		\begin{enumerate}
			\item A $d$-cokernel
			\[
			X_1 \xrightarrow{f_1} X_2 \xrightarrow{f_2} \dots \xrightarrow{f_{d-1}} X_{d} \xrightarrow{f_d} X_{d+1} \to 0
			\]
			of a morphism $f_0 \colon X_0 \to X_1$ in $\mc$ is called \textit{minimal} if $f_i \in \Rad_{\ac}(X_i,X_{i+1})$ for \mbox{$i=2,\ldots,  d$}.
			\item A $d$-kernel
			\[
			0 \to X_0 \xrightarrow{f_0} X_1 \xrightarrow{f_1} \dots \xrightarrow{f_{d-2}} X_{d-1} \xrightarrow{f_{d-1}} X_{d}
			\]
			of a morphism $f_d \colon X_d \to X_{d+1}$ in $\mc$ is called \textit{minimal} if $f_i \in \Rad_{\ac}(X_i,X_{i+1})$ for $i=0,\ldots,  d-2$.
			\item A $d$-extension
			\[
			0 \to X_0 \xrightarrow{f_0} X_1 \xrightarrow{f_1} \dots \xrightarrow{f_{d-1}} X_{d} \xrightarrow{f_d} X_{d+1}\to 0
			\]
			in $\mc$ is called \textit{minimal} if $f_i \in \Rad_{\ac}(X_i,X_{i+1})$ for $i=1,\ldots,  d-1$.
		\end{enumerate}
	\end{definition}
	
	\cref{prop: minimal} justifies the terminology in the definition above, and shows that minimal \mbox{$d$-cokernels}, $d$-kernels and $d$-extensions exist and are unique up to isomorphism. When we say that a complex in a category $\mc$ is a \textit{direct summand} of another complex, this means that it is a direct summand in the category of complexes over $\mc$.
	
	\begin{proposition}\cite[Proposition 2.4]{HerschendJorgensen} \label{prop: minimal}
		Let $\mc$ be a $d$-cluster tilting subcategory of $\ac$.
		\begin{enumerate}
			\item Given a morphism $f$ in $\mc$, there exists a minimal $d$-cokernel (resp.\ $d$-kernel) of $f$. This $d$-cokernel (resp.\ $d$-kernel) is a direct summand of any other $d$-cokernel (resp.\ $d$-kernel) of $f$. 
			\item Given a $d$-extension in $\mc$, there exists an equivalent minimal $d$-extension. This minimal $d$-extension is a direct summand of every $d$-extension in the associated equivalence class.
		\end{enumerate}
	\end{proposition}
	
	\begin{remark} \label{rem: minimal}
		Since \cite[Lemma 2.1]{Jasso} implies that any two $d$-cokernels of a morphism are isomorphic in the homotopy category, \cref{prop: minimal} implies that any $d$-cokernel is isomorphic to the direct sum of the minimal $d$-cokernel and a contractible exact sequence. Moreover, for an idempotent complete category (e.g. an abelian category), any contractible complex is the direct sum of complexes of the form $N \xrightarrow{1_N} N$. In particular, given a $d$-cokernel 
		\[
		X_1 \xrightarrow{f_1} \cdots \xrightarrow{f_{d-1}} X_{d} \xrightarrow{f_d} X_{d+1} \to 0
		\]
		of $f_0\colon X_0 \rightarrow X_1$ where $f_i \notin \Rad_\ac(X_i, X_{i+1})$ for some $2 \leq i\leq d$, there is an isomorphic $d$-cokernel where $f_i$ is replaced by
		\begin{align*}
			\begin{pmatrix}
				f_i' & 0 \\ 0 & 1_N
			\end{pmatrix}\colon X_{i}' \oplus N \rightarrow X_{i+1}' \oplus N 
		\end{align*}
		with $f_i' \in \Rad_\ac(X_i', X_{i+1}')$. Similar statements hold for $d$-kernels and $d$-extensions.
	\end{remark}
	
	The terminology in \cref{def: minimal} is further justified by the following connection to minimal morphisms.
	
	\begin{lemma}\label{rem: minimality and radical}
		Let $X \xrightarrow{f} Y \xrightarrow{g} Z$ be a complex in $\ac$ for which the induced sequence 
		\[
		\Hom_\ac(Z,Y) \to \Hom_\ac(Y,Y) \to \Hom_\ac(X,Y)
		\]
		is exact. The morphism $f$ is left minimal if and only if $g\in\Rad_{\ac}(Y,Z)$. 
	\end{lemma}
	
	\begin{proof}
		See \cite[Lemma 1.1]{higher survey}
	\end{proof}
	
	The construction of a minimal $d$-cokernel is frequently used throughout the paper. We  discuss it in more detail in the following.
	
	\begin{construction}\label{Example:MinimaldCokernel}
		The minimal $d$-cokernel of a morphism $f \colon X \rightarrow Y$ in a $d$-cluster tilting subcategory $\mc \subseteq \ac$ can be constructed as follows:
		\begin{enumerate}
			\item Set $C_1=\Coker f$ and let $g_1 \colon C_1 \to M_1$ be the minimal left $\mc$-approximation of $C_1$. Set $f_1 \colon Y \to M_1$ to be the composition $Y \rightarrow C_1 \xrightarrow{g_1} M_1$.
			\item Repeat on $f_1 \colon Y \to M_1$ to construct $f_2 \colon M_1 \to M_2$.
			\item Iterate the procedure, which must terminate and result in a $d$-cokernel by \cite[Proposition 3.17]{Jasso}.
		\end{enumerate}
		Since each morphism $f_i$ is the composition of an epimorphism and a left minimal morphism, they are all left minimal, and hence \cref{rem: minimality and radical} shows that this is the minimal $d$-cokernel of $f$. Note that each $f_i$ can equivalently be described as the left minimal weak cokernel of the previous morphism.
	\end{construction}
	
	\begin{lemma} \label{lem: minimal pushouts}
		Suppose we have a $d$-extension
		\begin{align*}
			0 \to X \rightarrow E_1 \rightarrow E_2 \rightarrow \dots \rightarrow E_d \rightarrow Y \to 0.
		\end{align*}
		Then, for any morphism $h \colon X \to F_0$, there is a $d$-pushout diagram
		\begin{center}
			\begin{tikzcd}[column sep=22, row sep=22]
				0 \arrow[r] & X \arrow[r] \arrow[d,"h"] & E_1 \arrow[r] \arrow[d] & \cdots \arrow[r] & E_d \arrow[r] \arrow[d] & Y\arrow[r] \arrow[d, equal] & 0 \\
				0 \arrow[r] & F_0 \arrow[r] & F_1 \arrow[r] & \cdots  \arrow[r] & F_d \arrow[r] & Y \arrow[r] & 0
			\end{tikzcd}
		\end{center}
		such that the bottom row is a minimal $d$-extension.
	\end{lemma}
	
	\begin{proof}
		The fact that a $d$-pushout diagram exists is precisely \cite[Theorem 3.8]{Jasso}.
		It follows from \cite[Proposition 4.8]{Jasso} that the bottom row is a $d$-extension with last term equal to $Y$. If this $d$-extension is minimal, then we are done. Otherwise, it has a minimal $d$-extension as a direct summand by \cref{prop: minimal}.
		Composing the morphism in the statement with the projection onto this minimal $d$-extension gives another commutative diagram, which is a $d$-pushout by \cite[Proposition 4.8]{Jasso}. The bottom row of this new diagram is a minimal $d$-extension as required.
	\end{proof}
	
	Note that there is a dual version of \cref{lem: minimal pushouts} which will also be used.
	
	\subsection{Torsion and \texorpdfstring{$d$}{d}-torsion classes}
	
	Generalising the properties of the class of torsion groups in the category of abelian groups, the notion of a \textit{torsion pair} was introduced in \cite{Dickson} as follows.
	
	\begin{definition}\label{def:torsion}
		A pair $(\tc, \fc)$ of subcategories of $\ac$ is a \textit{torsion pair} if the following conditions are satisfied:
		\begin{enumerate}
			\item For every $X \in \ac$, there exists a short exact sequence
			\[0\to tX \rightarrow X \rightarrow fX \to 0\]
			where $tX\in\tc$ and $fX\in\fc$.
			\item $\Hom_\ac (X,Y)=0$ for all $X \in \tc$ and $Y \in \fc$.
		\end{enumerate}
		Given a torsion pair $(\tc, \fc)$, we say that $\tc$ is a \textit{torsion class} and $\fc$ a \textit{torsion free class}.
	\end{definition}
	
	Note that the short exact sequence from the definition above is unique up to isomorphism. The following classical result characterises torsion classes in $\ac$ as those subcategories which are closed under extensions and quotients.
	
	\begin{theorem}\cite[Theorem 2.3]{Dickson} \label{thm: dickson torsion}
		A subcategory $\tc$ of $\ac$ is a torsion class if and only if $\tc$ is closed under extensions and quotients.
	\end{theorem} 
	
	We denote by $\text{tors}(\ac)$ the poset of all torsion classes in $\ac$ ordered by inclusion. It is well known that $\text{tors}(\ac)$ 
	is a complete lattice where the meet is given by intersection, see e.g.\ \cite[Proposition 2.3]{IyamaReitenThomasTodorov}. 
	
	In light of the development of higher homological algebra, it is natural to consider higher analogues of torsion classes. 
	The notion of $d$-torsion classes was introduced in \cite{Jorgensen}. 
	
	\begin{definition}\cite[Definition 1.1]{Jorgensen}\label{def:ntorsionclass}
		Let $\mc$ be a $d$-abelian category.
		A subcategory $\uc$ of $\mc$ is a \textit{$d$-torsion class} if for every $M$ in $\mc$, there exists a $d$-exact sequence
		\begin{equation*}
			0 \rightarrow U_M \rightarrow M \rightarrow V_1 \rightarrow \cdots \rightarrow V_d \rightarrow 0
		\end{equation*}
		such that the following conditions are satisfied:
		\begin{enumerate}
			\item \label{def:ntorsionclass:contained}The object $U_M$ is in $\uc$.
			\item \label{def:ntorsionclass:exact} The sequence $0 \rightarrow \Hom_\mc(U,V_1) \rightarrow \cdots \rightarrow \Hom_\mc(U,V_d) \rightarrow 0$ is exact for every $U$ in $\uc$.
		\end{enumerate}
		The object $U_M$ above is known as the \textit{$d$-torsion subobject} of $M$ with respect to $\uc$. 
	\end{definition}
	Note that when $d=1$, the previous definition is equivalent to the definition of a torsion class, c.f.\ \cref{def:torsion}.
	
	\begin{remark} \label{rem: closed under summands}
		Our standing assumption that subcategories are closed under finite direct sums and summands is not necessary in \Cref{def:ntorsionclass}, as this follows from \cite[Lemma 2.7(iii)]{Jorgensen}.
	\end{remark}
	
	A recent paper \cite{AJST} showed that there is a strong relationship between the torsion classes in an abelian category $\ac$ and the $d$-torsion classes in a $d$-cluster tilting subcategory $\mc$ of $\ac$.
	
	\begin{theorem}\cite[Theorem 1.1]{AJST}\label{thm:torsion-to-dtorsion}
		Let $\mc$ be a $d$-cluster tilting subcategory of $\ac$. Then a torsion class $\tc$ in $\ac$ is the minimal torsion class containing a given $d$-torsion class in $\mc$ if and only if the following conditions are satisfied:
		\begin{enumerate}
			\item For every $M\in\mc$, we have $tM\in\mc$.
			\item $\tc$ is the smallest torsion class containing all $tM$ for $M\in\mc$.
			\item For any $M,M'\in\mc$, we have $\Ext^{d-1}_\ac(tM,fM')=0$. 
		\end{enumerate}
		Moreover, in this case $\uc \colonequals \tc \cap \mc$ is a $d$-torsion class and $\tc$ is the minimal torsion class containing it. Furthermore, we have $U_M \cong tM$ for every object $M \in \mc$.
	\end{theorem}
	
	We now illustrate \cref{thm:torsion-to-dtorsion} in our running example. 
	
	\begin{example}\label{ex:running2} 
		Let $A$ and $\mc$ be as in \cref{ex:running}.
		In \cref{tab:intersection} we give the complete list of $2$-torsion classes $\uc$ in $\mc$ and the corresponding minimal torsion classes $\tc$ in $\modA$ such that $\uc=\tc \cap \mc$.
		
		%%%%%%%
		
		\begin{table}[ht]
			\centering
			\begin{tabular}{|c|c|}
				\hline
				$2\text{-torsion classes in } \mc $ & $\text{Corresponding minimal torsion classes in } \modA        $  \\
				\hline
				\hline
				$\mc$ & $\modA$\\
				\hline
				$\add\left\{\rep{2\\3}\oplus \rep{1\\2} \oplus \rep{1}\right\}$ & $\add\left\{\rep{2\\3}\oplus \rep{2}\oplus \rep{1\\2} \oplus \rep{1}\right\}$
				\\
				\hline
				$\add\left\{\rep{1\\2} \oplus \rep{1}\right\}$ & $\add\left\{ \rep{1\\2} \oplus \rep{1}\right\}$
				\\
				\hline
				$\add\left\{\rep{1}\right\}$ & $\add\left\{\rep{1}\right\}$
				\\
				\hline
				$\add\left\{\rep{3}\right\}$ & $\add\left\{\rep{3}\right\}$\\
				\hline
				$\left\{0\right\}$ & $\left\{0\right\}$ \\
				\hline
			\end{tabular}
			\caption{The $2$-torsion classes considered in \cref{ex:running2} and their corresponding torsion classes.}
			\label{tab:intersection}
		\end{table}
		
		%%%%%%%%%%%%%%%%%%%%%%%%%%%%%%%
		
		On one hand, we note that although \cref{thm:torsion-to-dtorsion} gives a complete characterisation of the {\it minimal} torsion classes $\tc$ in $\modA$ such that $\tc\cap\mc$ is a $2$-torsion class in $\mc$, there are other torsion classes in $\modA$ with this property. For instance, we see that $\tc=\add\left\{\rep{2}\oplus\rep{1\\2}\oplus\rep{1} \right\}$ is a torsion class in $\modA$ and $\tc\cap\mc = \add\left\{\rep{1\\2}\oplus\rep{1} \right\}$ is a $2$-torsion class in $\mc$. 
		
		On the other hand, it is not true that $\tc\cap\mc$ is a \mbox{$2$-torsion} class in $\mc$ for every torsion class $\tc$ in $\modA$. Two such examples are $\tc_1=\add\left\{\rep{3}\oplus\rep{1} \right\}$ and $\tc_2 = \add\left\{\rep{3}\oplus\rep{2\\3}\oplus\rep{2} \right\}$. The intersections $\tc_1\cap\mc$ and $\tc_2\cap \mc$ are not $2$-torsion classes in $\mc$, for reasons that will be explained in \cref{ex:running3}.
	\end{example} 
	
	%%%%%%%%%%%%%%%%%%%%%%%%%%%%%%%
	
	\section{Closure under \texorpdfstring{$d$}{d}-extensions and \texorpdfstring{$d$}{d}-quotients}
	\label{sec: general results}
	Throughout this section, let $\mc$ be a $d$-cluster tilting subcategory of
	the abelian category $\ac$. We also introduce the following setup, which will not be assumed unless explicitly stated.
	
	\begin{setup} \label{torsion setup}
		Let $\uc$ be a $d$-torsion class in $\mc$. Let $\tc$ be a torsion class in $\ac$ such that $\uc = \tc \cap \mc$ and the torsion subobject $tM$ of $M$ belongs to $\mc$ for all $M \in \mc$.
	\end{setup}
	
	We note that \cref{thm:torsion-to-dtorsion} implies that for every $d$-torsion class $\uc$ in $\mc$, we can find a torsion class $\tc$ satisfying \cref{torsion setup}.
	
	\cref{sec: general results} is divided into three subsections. We first verify  some consequences of the setup above in \cref{Subsection:Approx}. In \cref{subsec:characterising} we state and prove the main result of this paper, namely the characterisation of higher torsion classes given in \cref{intro: characterisation}. The aim of \cref{subsec: closure by indecs} is to prove \cref{intro: indec}. 
	
	\subsection{Approximations by $d$-torsion classes}\label{Subsection:Approx}
	Throughout this subsection we assume \cref{torsion setup}.
	
	\begin{lemma}\label{lem:sametorsionobject}
		For any $M \in \mc$, there is an isomorphism $U_M \cong tM$ commuting with the inclusion to $M$.
	\end{lemma}
	
	\begin{proof}
		By the definition of a $d$-torsion class, the morphism $U_M\to M$ is a right $\uc$-approximation. Since the morphism $tM\to M$ is a right $\tc$-approximation and $tM\in \uc$, it must also be a right $\uc$-approximation. Hence, the inclusions $U_M\to M$ and $tM\to M$ must factor through each other, which implies that $U_M\cong tM$.
	\end{proof}
	
	By \cref{lem:sametorsionobject}, we can assume $U_M=tM$ whenever we are in \cref{torsion setup}, and we will do this from now on. Given \cref{torsion setup}, we gain additional control of left $\mc$-approximations of objects in $\tc$.
	
	\begin{lemma} \label{lemma: minimal M-appr. in U}
		Let $\phi_X\colon X \rightarrow M$ be the minimal left $\mc$-approximation of an object $X$ in $\tc$. Then $M$ is in $\uc$.
	\end{lemma}
	
	\begin{proof}
		Consider the short exact sequence
		\[
		0 \rightarrow tM \xrightarrow{\iota} M \rightarrow f M \rightarrow 0
		\]
		associated to the torsion pair $(\tc,\fc)$. The morphism $\iota$ is a right $\tc$-approximation of $M$, and $tM \in \mc$ by assumption. As $X \in \tc$, there exists a morphism $\psi_X \colon X \to tM$ such that $\iota \psi_X = \phi_X$. Notice that since $\phi_X$ is a left $\mc$-approximation of $X$, so is $\psi_X$. Since $\phi_X$ is left minimal and $\iota$ is a monomorphism, this implies that $M$ is isomorphic to $tM$. Consequently, one obtains that $M$ is contained in $\tc$, and hence also in $\uc$.
	\end{proof}
	
	\begin{corollary}\label{cor: minimal M-appr. in U}
		If $X \in\tc$, then the minimal left \mbox{$\mc$-approximation} of $X$ is also the minimal left $\uc$-approximation of $X$. Moreover, this approximation is a monomorphism.
	\end{corollary}
	
	\begin{proof}
		Let $\phi_X\colon X \rightarrow M$ be the minimal left \mbox{$\mc$-approximation} of $X$. By \cref{lemma: minimal M-appr. in U}, we know that $M \in \uc$, and it follows that $\phi_X$ is also a left $\uc$-approximation. Since $\phi_X$ is left minimal, it is the minimal left $\uc$-approximation of $X$. Finally, since $\mc$ is cogenerating, any left \mbox{$\mc$-approximation} is a monomorphism.
	\end{proof}
	
	\begin{remark}\label{rem: U contr. finite in T}
		Note that the previous corollary implies that the $d$-torsion class $\uc$ is always covariantly finite within the torsion class $\tc$, even if $\uc$ is not covariantly finite in $\mc$ or $\ac$.
	\end{remark}
	
	\subsection{Characterising \texorpdfstring{$d$}{d}-torsion classes}
	\label{subsec:characterising}
	
	In order to formulate our results, we need higher analogues of what it means for a subcategory to be closed under extensions and quotients. Recall the notions of $d$-cokernels and $d$-extensions from  \cref{subsec:d-abelian categories}.
	
	The following definition already appeared in the literature, see e.g. \cite[Definition 2.8(iii)]{HerschendJorgensenVaso} or \cite[Definition 4.1]{HerschendLiuNakaoka2}.
	\begin{definition} \label{def: d-ext closed}
		A subcategory $\uc$ of $\mc$ is called \textit{closed under $d$-extensions} if for any $d$-extension 
		\[
		0 \to X \to E_1 \to \cdots \to E_d \to Y \to 0
		\]
		in $\mc$ with $X$ and $Y$ in $\uc$, there exists an equivalent $d$-extension
		\[
		0 \to X \to E'_1 \to \cdots \to E'_d \to Y \to 0
		\]
		where all the objects are in $\uc$.
	\end{definition}
	
	\begin{definition} \label{def: dq closure}
		A subcategory $\uc$ of $\mc$ is called \textit{closed under $d$-quotients} if for any morphism $f \colon M \rightarrow U$ in $\mc$ with $U$ in $\uc$, there exists a $d$-cokernel 
		\[
		M \xrightarrow{f} U \rightarrow E_1 \rightarrow E_2 \rightarrow \dots \rightarrow E_{d} \to 0
		\]
		of $f$ with $E_i$ in $\uc$ for all $i = 1,\ldots, d$.
		If this condition is only assumed to hold when both $M$ and $U$  belong to $\uc$, we say that $\uc$ is \textit{closed under $d$-cokernels}.    
	\end{definition}
	
	It will further be convenient to define a \textit{(minimal) $d$-quotient of} $U \in \mc$ as a (minimal) \mbox{$d$-cokernel} of some morphism $f \colon M \to U$ in $\mc$.  The following lemma shows that closure under $d$-quotients or $d$-extensions is equivalent to closure under minimal $d$-quotients or $d$-extensions. This uses the standing assumption that our subcategories are closed under direct summands.
	
	\begin{lemma}\label{Lemma:ClosureMinimality}
		Let $\uc$ be a subcategory of $\mc$. The following hold:
		\begin{enumerate}
			\item\label{Lemma:ClosureMinimality:1} $\uc$ is closed under $d$-extensions if and only if for any minimal $d$-extension
			\[
			0 \to X \to E_1 \to \cdots \to E_d \to Y \to 0
			\]
			with $X, Y \in \uc$, we have $E_1, \dots, E_d \in \uc$.
			\item\label{Lemma:ClosureMinimality:2} $\uc$ is closed under $d$-quotients if and only if for any minimal $d$-cokernel
			\[
			M \xrightarrow{} U \rightarrow E_1 \rightarrow E_2 \rightarrow \dots \rightarrow E_{d} \to 0
			\]
			of a morphism $M \to U$ in $\mc$
			with $U \in \uc$, we have $E_1, \dots, E_d \in \uc$.
		\end{enumerate}
	\end{lemma}
	
	\begin{proof}
		By \Cref{prop: minimal}, any equivalence class of $d$-extensions contains a unique (up to isomorphism) minimal $d$-extension, which is furthermore a direct summand of any other $d$-extension in the class. This immediately implies \eqref{Lemma:ClosureMinimality:1}, since $\uc$ is closed under direct summands. Part \eqref{Lemma:ClosureMinimality:2} is proved similarly, using that the minimal $d$-cokernel of a morphism $f$ is a summand of any other $d$-cokernel of $f$.
	\end{proof}
	
	\begin{remark} \label{rem: closed under summands 2}
		For a direct sum $X \oplus Y \in \mc$, we have $X,Y\in\mc$ and the projection $X \oplus Y \to Y$ gives a minimal $d$-cokernel
		\[
		X\to X\oplus Y\to Y \to 0\to \cdots \to 0
		\]
		of the inclusion $X \to X \oplus Y$.
		
		Hence, being closed under minimal $d$-quotients implies being closed under direct summands. 
	\end{remark}
	
	\cref{rem: closed under summands} and \cref{rem: closed under summands 2} show that the standing assumption on subcategories being closed under direct summands is not a significant restriction.
	
	The following lemma shows that when checking if a subcategory is closed under $d$-extensions, it is often sufficient to consider the first middle term. This simplifies the proof of \cref{theorem: extension closed} and is also an important step towards the main result in \cref{subsec: closure by indecs}.
	
	\begin{lemma} \label{lem:second term}
		Let $\uc \subseteq \mc$ be closed under $d$-quotients. Then for any minimal $d$-extension
		\begin{align*}
			0 \to X \xrightarrow{f} E_1 \xrightarrow{e_1} E_2 \xrightarrow{e_2} \cdots \xrightarrow{e_{d-1}} E_d \xrightarrow{g} Y \to 0
		\end{align*}
		in $\mc$ with $X,E_1,Y \in \uc$, it follows that $E_i \in \uc$ for $i=2, \dots, d$.
	\end{lemma}
	\begin{proof} 
		The minimality of the $d$-extension (see \cref{def: minimal})
		gives that $e_i \in \Rad_\ac(E_i, E_{i+1})$ for all $i=1, \dots, d-1$. Moreover, the sequence
		\begin{align}
			E_1 \xrightarrow{e_1} E_2 \xrightarrow{e_2} \cdots \xrightarrow{e_{d-1}} E_d \xrightarrow{g} Y \to 0 \label{eq: secondterm}
		\end{align}
		is a $d$-cokernel of $f$. If $g \in \Rad_\ac(E_d,Y)$, then this $d$-cokernel is minimal. Since $E_1 \in \uc$, the result then follows from $\uc$ being closed under minimal $d$-quotients by \Cref{Lemma:ClosureMinimality}. 
		
		Suppose hence that $g \notin \Rad_\ac(E_d,Y)$. Recall from \cref{rem: minimal} that the sequence \eqref{eq: secondterm} is isomorphic to the direct sum of the minimal $d$-cokernel of $f$ and shifted complexes of the form $N \xrightarrow{1_N} N$. However, since $e_i \in \Rad_A(E_i,E_{i+1})$ for $i=1, \dots d-1$, it follows that \eqref{eq: secondterm} is isomorphic to 
		\begin{align*}
			E_1 \xrightarrow{e_1} E_2 \xrightarrow{e_2} \cdots \xrightarrow{e_{d-2}} E_{d-1} \xrightarrow{e_{d-1}'=\begin{psmallmatrix} h_1 \\ 0 \end{psmallmatrix}} E'_d \oplus Y'' \xrightarrow{\begin{psmallmatrix} g' & 0 \\ 0 & 1_{Y''} \end{psmallmatrix}} Y' \oplus Y'' \to 0,
		\end{align*}
		where $e_{d-1}'$ is $e_{d-1}$ composed with an isomorphism and 
		\begin{align*}
			E_1 \xrightarrow{e_1} E_2 \xrightarrow{e_2} \cdots \xrightarrow{e_{d-2}} E_{d-1} \xrightarrow{h_1} E'_d \xrightarrow{ g'} Y' \to 0
		\end{align*}
		is a minimal $d$-cokernel of $f$. In particular, the objects $E_2, \dots, E_{d-1},E_d'$ are in $\uc$ as $E_1\in\uc$ and $\uc$ is closed under minimal $d$-quotients by \Cref{Lemma:ClosureMinimality}. Since $Y'\oplus Y'' = Y \in \uc$ and $\uc$ is closed under direct summands, we have that $Y''\in \uc$. Hence 
		$E_d=E_d' \oplus Y''\in\uc$, completing the proof.
	\end{proof}
	
	We are now ready to prove the first part of our characterisation result. 
	
	\begin{proposition} \label{theorem: extension closed}
		Let $\uc \subseteq \mc$ be a $d$-torsion class. Then $\uc$ is closed under $d$-extensions and $d$-quotients.
	\end{proposition}
	
	\begin{proof}
		As $\uc$ is a $d$-torsion class in $\mc$, we have that $\uc = \tc \cap \mc$ for some torsion class $\tc$ in $\ac$ as in \cref{torsion setup}. Recall from \cref{Lemma:ClosureMinimality} that it suffices to consider minimal $d$-extensions and minimal $d$-quotients.
		
		We first show that $\uc$ is closed under minimal $d$-quotients. Let $f \colon M \to U$ be a morphism in $\mc$ with $U \in \uc$ and consider its minimal $d$-cokernel
		\[
		M \xrightarrow{f} U \xrightarrow{f_1} V_1 \xrightarrow{f_2} \cdots \xrightarrow{f_{d-1}} V_{d-1}\xrightarrow{f_{d}} V_{d}\to 0.
		\]
		By construction of the minimal $d$-cokernel (see \cref{Example:MinimaldCokernel}), we have that $V_i$ arises as the minimal left $\mc$-approximation of $\Coker f_{i-1}$ for all $i= 1,\ldots,  d$, where we set $f_0 = f$. As $U \in \uc = \tc \cap \mc$ and $\tc$ is closed under quotients, \cref{lemma: minimal M-appr. in U} implies that $V_i \in \uc$ for all $i= 1,\ldots,  d$.
		This shows that $\uc$ is closed under minimal $d$-quotients. 
		
		We next prove that $\uc$ is closed under minimal $d$-extensions. Consider a minimal $d$-exact sequence 
		\[
		0 \to X \xrightarrow{e_0} E_1 \xrightarrow{e_1} \cdots \xrightarrow{e_{d-1}} E_d \xrightarrow{e_d} Y \to 0
		\]
		in $\mc$ with $X$ and $Y$ in $\uc$. 
		By \cref{rem: minimality and radical}, the morphism $e_i$ is left minimal for $i=0,\dots,d-2$. As $\uc$ is a $d$-torsion class, we obtain the solid part of the diagram
		\[
		\begin{tikzcd}[column sep=22, row sep=30]
			0 \arrow[r] & X \arrow[r,"e_0"] \arrow[d, dashed, "g_0"] & E_1 \arrow[r, "e_1"] \arrow[d, equal] \arrow[dl, dashed, "h_0"] & E_2 \arrow[r, "e_2"] \arrow[d, dashed, "g_2"] \arrow[dl, dashed, "h_1"] & \cdots \arrow[r, "e_{d-1}"] & E_d \arrow[r, "e_d"] \arrow[d, dashed, "g_d"] \arrow[dl, dashed, "h_{d-1}"] & Y \arrow[r] \arrow[d, dashed, "g_{d+1}"] \arrow[dl, dashed, "h_d"] & 0 \\
			0 \arrow[r] & tE_1 \arrow[r, "\iota"] & E_1 \arrow[r, "w_0"] & W_1 \arrow[r, "w_1"] & \cdots \arrow[r, "w_{d-2}"] & W_{d-1} \arrow[r, "w_{d-1}"] & W_d \arrow[r] & 0,
		\end{tikzcd}
		\]
		where the bottom row is the $d$-exact sequence associated to $E_1$ by $\uc$. In particular, the object $tE_1 \in \uc$ is the $d$-torsion subobject of $E_1$ with respect to $\uc$ by \cref{lem:sametorsionobject}, 
		and the sequence 
		\begin{equation}\label{Equation:AvoidingUExact}0 \rightarrow \Hom_\ac(U,W_1) \rightarrow \cdots \rightarrow \Hom_\ac(U,W_d) \rightarrow 0
		\end{equation}
		is exact for every $U$ in $\uc$. As $\iota$ is a right $\tc$-approximation and $X \in \uc = \tc \cap \mc$, there exists a morphism $g_0 \colon X \to tE_1$ making the left square commute. We can hence complete the diagram to a morphism $g$ of $d$-exact sequences by using the factorisation property for weak cokernels, see \cref{subsec:d-abelian categories}.
		
		Since \eqref{Equation:AvoidingUExact} is exact and $Y\in \uc$, the morphism \[w_{d-1} \circ - \colon \Hom_\ac(Y,W_{d-1}) \rightarrow \Hom_\ac(Y,W_d)\] is surjective. Hence, there exists a morphism $h_d\colon Y \to W_{d-1}$ with $w_{d-1} h_d=g_{d+1}$. As the bottom row is $d$-exact, there is an exact sequence
		\[
		\Hom_\ac(E_d,W_{d-2}) \rightarrow \Hom_\ac(E_d, W_{d-1}) \rightarrow \Hom_{\ac}(E_d,W_d).
		\]
		Using the commutativity of the rightmost square, we get $w_{d-1}(g_d-h_de_d)=0$, so \mbox{$g_d-h_de_d$} is in the kernel of the second morphism. By exactness, there exists a morphism \mbox{$h_{d-1}\colon E_d \to W_{d-2}$} such that $g_d-h_de_d = w_{d-2}h_{d-1}$, or equivalently $g_d= h_d e_d + w_{d-2}h_{d-1}$.
		
		We can repeat this process to obtain a homotopy of the map of complexes $g$. In particular, there are morphisms $h_0$ and $h_1$ such that $g_1 =1_{E_1} = \iota h_0 + h_1 e_1$. This implies that $\iota h_0 e_0 = e_0$, so $\iota h_0$ is an isomorphism by the left minimality of $e_0$. The morphism $h_0$ is hence a split monomorphism, so $E_1$ is contained in $\uc$. By \cref{lem:second term}, this implies that $E_i \in \uc$ for \mbox{$i=2,\dots,d$}, so the subcategory $\uc$ is closed under minimal $d$-extensions.
	\end{proof}
	
	\begin{remark}
		\cref{theorem: extension closed} implies that every $d$-torsion class is closed under $d$-cokernels.
	\end{remark}
	
	The remainder of this subsection is devoted to proving the converse of \cref{theorem: extension closed}. Since any $d$-torsion class in $\mc$ is contravariantly finite in $\mc$, we first establish that for a subcategory $\uc \subseteq \mc$ to be contravariantly finite, it is enough to assume closure under $d$-quotients.
	
	We first need a result on coimage factorisation in a $d$-cluster tilting subcategory. Recall from \cref{subsec:d-abelian categories} that a morphism $g$ in $\mc$ is a weak cokernel if there exists a morphism $f$ in $\mc$ such that $g$ is a weak cokernel of $f$.

	\begin{proposition}\label{Lemma: Coimage factorization}
		Let $f\colon M\to N$ be a morphism in $\mc$. Then there exists a factorisation $f=f_2\circ f_1$ in $\mc$ where $f_2$ is a monomorphism and $f_1$ is a composite of left minimal weak cokernels. 
	\end{proposition}
	
	\begin{proof}
		Let $E_0$ denote the image of $f$, and let $\iota_0\colon E_0\to M_0$ be a minimal left $\mc$-approximation. 
		The inclusion $E_0\to N$ lifts via $\iota_0$ to a morphism $g_0\colon M_0\to N$. Let $E_1$ be the image of $g_0$, and note that $E_0\subseteq E_1$. Repeating this procedure, we get a subobject $E_i\subseteq N$, a minimal left approximation $\iota_i\colon E_i\to M_i$, and a lift $g_i\colon M_i\to N$ for each $i\geq 0$. In particular, the $E_i$'s form an increasing sequence  $E_0\subseteq E_1\subseteq E_2\subseteq \cdots \subseteq E_i\subseteq \cdots$ of subobjects of $N$. 
		
		Since $\ac$ is of finite length, this sequence has to stabilise, say at $E_j$. This implies that the image of $g_j\colon M_j\to N$ is $E_j$. But then the inclusion $\iota_j\colon E_j\rightarrow M_j$ is a split monomorphism, and hence an isomorphism since it is also left minimal.  This shows that $E_j\in \mc$. Now let $f_2\colon E_j\to N$ be the inclusion, 
		and let $f_1\colon M\to E_j$ be the composite $M\to E_0\rightarrow E_j$.
		Note that this is equal to the composite
		\[
		\begin{tikzcd}[column sep=0.35cm, row sep=0.5cm]
			M \arrow[two heads]{rd}\arrow{rr}{}  && M_0 \arrow{rr}{} \arrow[two heads]{rd} && M_1 \arrow{rr}{} \arrow[two heads]{rd} &&  \cdots \arrow{rr}{} \arrow[two heads]{rd} && M_{j-1} \arrow{rr}{} \arrow[two heads]{rd} && M_j. \\
			& E_0 \arrow[tail]{ru}{\iota_0} \arrow[tail]{rr} && E_1 \arrow[tail]{ru}{\iota_1} \arrow[tail]{rr} &&  E_2 \arrow[tail]{ru}{\iota_2} & \cdots & E_{j-1} \arrow[tail]{ru}{\iota_{j-1}} \arrow[tail]{rr} && E_j\arrow[tail, "\cong"]{ru} & 
		\end{tikzcd}
		\]
		By construction, the morphisms $M\to M_0$ and $M_i\to M_{i+1}$ for $i=0,\ldots, j-1$ are left minimal weak cokernels. Since $f_1$ is a composite of such morphisms and the isomorphism $M_j\cong E_j$, this proves the claim.
	\end{proof}
	
	\begin{lemma}\label{Lemma:MinWeakCoker}
		Let $\uc\subseteq \mc$ be closed under $d$-quotients. If $g\colon M\to N$ is a left minimal weak cokernel with $M\in \uc$, then $N\in \uc$.
	\end{lemma}
	\begin{proof}
		Assume $g$ is a weak cokernel of a morphism $f$. Then $g$ is part of a $d$-cokernel of $f$, and if $g$ is left minimal, then it is part of the minimal $d$-cokernel of $f$; see \cref{Example:MinimaldCokernel}. This proves the claim.
	\end{proof}
	
	We now apply \cref{Lemma: Coimage factorization,Lemma:MinWeakCoker} to show that being closed under $d$-quotients implies being contravariantly finite.
	
	\begin{proposition} \label{prop: contravariantly finite}
		Let $\uc \subseteq \mc$ be closed under $d$-quotients and consider $M\in \mc$. Then there exists a minimal right $\uc$-approximation $U\to M$ which is a monomorphism. In particular, the subcategory $\uc$ is contravariantly finite in $\ac$.
	\end{proposition}
	
	\begin{proof}
		Consider the set  of subobjects $U\subseteq M$ with $U\in \uc$. Note that this set is non-empty since $0\in \uc$.
		
		We first prove that this set has a unique maximal element. Indeed, since $\ac$ has finite length, we can choose $U\subseteq M$ with $U\in \uc$ and where $U$ is of maximal length with this property. Now let $V\subseteq M$ with $V\in \uc$, and consider the induced morphism
		\(
		U\oplus V\to M.
		\)
		By \cref{Lemma: Coimage factorization}, there exists $W\subseteq M$ with $W\in \mc$ such that $U\subseteq W\subseteq M$ and $V\subseteq W\subseteq M$, and such that the induced morphism $U\oplus V\to W$ is a composite of left minimal weak cokernels. Since $\uc$ is closed under $d$-quotients, it follows from \cref{Lemma:MinWeakCoker} that $W\in \uc$. But since $U$ is maximal with respect to the property that $U\in \uc$ and $U\subseteq M$, it follows that $U=W$. This implies that $V$ must be contained in $U$, and hence $U$ is the unique maximal subobject $U\subseteq M$ satisfying $U\in \uc$. 
		
		Now let $U'\to M$ be an arbitrary morphism with $U'\in \uc$. By \cref{Lemma: Coimage factorization}, there exists an object $V'$ such that $U'\to M$ factors through $V'$ and where the morphism $U'\rightarrow V'$ is a composite of left minimal weak cokernels and the morphism $V'\rightarrow M$ is a monomorphism. It follows that $V'\in \uc$ and $V'$ is a subobject of $M$.
		
		By the maximality of $U$, we get that $V'\subseteq U\subseteq M$, so the morphism $U'\to M$ also factors through $U$. This shows that the inclusion $U\subseteq M$ is a right $\uc$-approximation, which is minimal since it is a monomorphism. This proves the first claim. 
		
		Finally, the fact that  $\uc$ is contravariantly finite in $\ac$ follows from the fact that $\mc$ is contravariantly finite in $\ac$ and $\uc$ is contravariantly finite in $\mc$.
	\end{proof}
	
	Using the proposition above, we now show the converse of \cref{theorem: extension closed}.
	
	\begin{proposition}\label{prop:extension closed gives torsion}
		Let $\uc \subseteq \mc$ be closed under $d$-extensions and $d$-quotients.
		Then $\uc$ is a $d$-torsion class in $\mc$.
	\end{proposition}
	
	\begin{proof}
		Consider an object $M \in \mc$. By the definition of a $d$-torsion class (see \cref{def:ntorsionclass}), we need to show that there exists a $d$-exact sequence 
		\[
		0 \rightarrow U_M \rightarrow M \rightarrow V_1 \rightarrow \cdots \rightarrow V_d \rightarrow 0
		\]
		where $U_M\in \uc$ and the sequence $0 \rightarrow \Hom_\mc(U,V_1) \rightarrow \cdots \rightarrow \Hom_\mc(U,V_d) \rightarrow 0$ is exact for every $U$ in $\uc$.
		
		Since $\uc$ is closed under $d$-quotients, \cref{prop: contravariantly finite} shows that we may take a minimal right $\uc$-approximation $f \colon U_M \to M$ which is a monomorphism.
		Taking the minimal $d$-cokernel of $f$ gives a $d$-exact sequence 
		\begin{equation}\label{eq:torsion sequence}
			0 \rightarrow U_M \xrightarrow f M \rightarrow V_1 \rightarrow \cdots \rightarrow V_d \rightarrow 0. 
		\end{equation}
		
		Let $U\in\uc$. As $f$ is a right $\uc$-approximation, we know that $\Hom_\mc(U,f)$ is an epimorphism. Thus, it follows from $d$-exactness of the sequence (\ref{eq:torsion sequence}) that
		\[
		0  \rightarrow \Hom_\mc(U,V_1) \rightarrow \cdots \rightarrow \Hom_\mc(U,V_{d-1}) \rightarrow \Hom_\mc(U,V_d)
		\]
		is exact. To finish our proof, we hence need to show that the rightmost morphism in this sequence is an epimorphism.
		
		Consider $h_d\in \Hom_\mc(U,V_d)$ and take a $d$-pullback of (\ref{eq:torsion sequence}) along $h_d$. This yields a commutative diagram
		\[
		\begin{tikzcd}
			0 \arrow[r] & U_M \arrow[r,"f_0"] & W_0 \arrow[r] \arrow[d,"h_0"] &W_1 \arrow[r] \arrow[d] & \cdots \arrow[r] & W_{d-1} \arrow[r,"f_d"] \arrow[d] & U \arrow[r] \arrow[d,"h_d"] & 0\\ 0 \arrow[r]& U_M \arrow[r, "f"]\arrow[u,equal]& M\arrow[r ] & V_1\arrow[r] & \cdots\arrow[r] & V_{d-1}\arrow[r]& V_d \arrow[r] & 0, \end{tikzcd}
		\]
		where the upper row is a $d$-exact sequence. By the dual of \cref{lem: minimal pushouts}, this $d$-extension can be assumed to be minimal, and then closedness of $\uc$ under $d$-extensions implies $W_i\in\uc$ for all $i=0,\ldots,d-1$ by \cref{Lemma:ClosureMinimality}.
		As $f$ is a right $\uc$-approximation, the morphism $h_0$ factors through $f$, so $f_0$ is a split monomorphism. It follows, by \cite[Proposition 2.6]{Jasso} and its dual, that $f_d$ is a split epimorphism, and hence $h_d$ factors through $V_{d-1}$. In particular, the morphism $\Hom_\mc(U,V_{d-1}) \rightarrow \Hom_\mc(U,V_d)$ is an epimorphism as required.
	\end{proof}
	
	We can now generalise the classical characterisation of torsion classes, cf.\ \cref{thm: dickson torsion}. Recall our standing assumption that subcategories are closed under direct summands.
	
	\begin{theorem}\label{cor: characterisation}
		Let $\mc$ be a $d$-cluster tilting subcategory of $\ac$. A subcategory $\uc \subseteq \mc$ is a $d$-torsion class if and only if it is closed under both $d$-extensions and $d$-quotients.
	\end{theorem}
	
	\begin{proof}
		The necessity follows from \cref{theorem: extension closed}, while the sufficiency follows from \cref{prop:extension closed gives torsion}. 
	\end{proof}
	
	We now demonstrate the use of \cref{cor: characterisation} in our running example.
	
	\begin{example}\label{ex:running3}
		Let $A$ and $\mc$ be as described in \cref{ex:running}. In \cref{ex:running2} we claimed that  \mbox{$\tc_1= \add\left\{\rep{3}\oplus\rep{1} \right\}$} and $\tc_2 = \add\left\{\rep{3}\oplus\rep{2\\3}\oplus\rep{2} \right\}$ are torsion classes  
		for which \mbox{$\tc_1\cap\mc= \add\left\{\rep{3}\oplus\rep{1} \right\}$} and $\tc_2\cap \mc = \add\left\{\rep{3}\oplus\rep{2\\3}\right\}$ are not $2$-torsion classes in $\mc$. We now use \cref{cor: characterisation} to explain why this is the case.
		
		Consider the exact sequence
		\[
		0 \longrightarrow \Rep{3} \longrightarrow \Rep{2\\3} \longrightarrow \Rep{1\\2} \longrightarrow \Rep{1} \longrightarrow 0.
		\]
		It is straightforward to check that this is a minimal $2$-extension in $\mc$. This implies that $\tc_1\cap \mc$ is not closed under $2$-extensions, so it is not a $2$-torsion class in $\mc$ by \cref{cor: characterisation}. 
		
		Similarly, using the same sequence, one can see that $\tc_2 \cap \mc$ is not closed under $2$-quotients.
		Therefore, \cref{cor: characterisation} implies that $\tc_2\cap \mc$ is not a $2$-torsion class.
	\end{example}
	
	A $d$-exact category is a pair $(\cc,\xc)$ consisting of an additive category $\cc$ and a class $\xc$ of $d$-exact sequences in $\cc$ satisfying certain axioms, see \cite[Definition 4.2]{Jasso}. One immediate consequence of our characterisation result is that any $d$-torsion class $\uc$ in $\mc$ carries the structure of a $d$-exact category. 
	
	\begin{corollary}\label{d-torsion is d-exact}
		Let $\uc \subseteq \mc$ be a $d$-torsion class. Consider the class $\xc$ of $d$-exact sequences in $\mc$ where all the terms are in $\uc$. Then $(\uc,\xc)$ is a $d$-exact category. 
	\end{corollary}
	
	\begin{proof}
		By \cref{cor: characterisation}, the subcategory $\uc$ is closed under $d$-extensions in $\mc$. The result hence follows by applying \cite[Corollary 4.15]{Klapproth}.
	\end{proof}
	
	Note that when viewing $\mc$ and $\uc$ as $d$-exangulated categories, see \cite{HerschendLiuNakaoka}, \cref{d-torsion is d-exact} moreover implies that $\uc$ is a $d$-exangulated subcategory of $\mc$ in the sense of \cite[Definition 3.7]{Haugland}. 
	
	\subsection{Closure under \texorpdfstring{$d$}{d}-extensions}\label{subsec: closure by indecs} 
	To check if a subcategory $\uc \subseteq \mc$ is closed under $d$-extensions, it is necessary to determine the middle terms of any minimal $d$-extension between any two (not necessarily indecomposable) objects in $\uc$. In this subsection we show that under certain conditions, it is enough to understand the $d$-extensions between indecomposable objects. 
	The main result is the following. 
	
	\begin{theorem} \label{prop: extension closure indec}
		Suppose $\uc \subseteq \mc$ is closed under $d$-extensions with indecomposable end terms and all $d$-quotients. Then $\uc$ is closed under all $d$-extensions.
	\end{theorem}
	
	We apply the theorem above in \cref{sec: higher AA,sec:Higher NA}, where we use it to give a combinatorial description of $d$-torsion classes of higher Auslander algebras of type $\mathbb A$ and higher Nakayama algebras of type $\mathbb{A}$ and $\mathbb{A}_\infty^\infty$. 
	
	In order to prove \cref{prop: extension closure indec}, recall first from \cref{Lemma:ClosureMinimality} that we may focus our attention purely on minimal $d$-extensions and minimal $d$-quotients. Our first step is to show that when closing a subcategory under $d$-extensions, it may suffice to consider $d$-extensions where the first term is indecomposable.
	
	\begin{lemma}\label{lem: indec first term}
		Suppose $\uc \subseteq \mc$ is closed under $d$-quotients. If $\uc$ is closed under $d$-extensions with indecomposable first term, then $\uc$ is closed under all $d$-extensions. 
	\end{lemma}
	
	\begin{proof}
		Assume that $\uc$ is closed under $d$-quotients and under $d$-extensions with indecomposable first term. Let 
		\begin{equation}\label{Equation:OriginalExtension}
			0\rightarrow X\rightarrow E_1 \rightarrow \cdots \rightarrow E_d \rightarrow Y \rightarrow 0
		\end{equation}
		be a minimal $d$-extension with $X,Y\in\uc$. We want to show that $E_i \in \uc$ for all $i=1,\dots,d$. By \cref{lem:second term}, it is sufficient to check that $E_1\in\uc$. 
		
		If $X$ is indecomposable, we are done by assumption. Suppose hence that $X=X_1\oplus X_2$, where $X_1$ is indecomposable and $X_2 \neq 0$. We take a $d$-pushout of the $d$-extension along the projection $\pi\colon X\rightarrow X_1$. This yields a commutative diagram
		\[
		\begin{tikzcd}
			0 \arrow[r] & X \arrow[r]\arrow[d] & E_1 \arrow[r]\arrow[d] & \cdots\arrow[r] & E_d\arrow[r]\arrow[d] & Y \arrow[r]\arrow[d, equal]& 0\\
			0 \arrow[r] & X_1 \arrow[r] & F_1 \arrow[r] & \cdots\arrow[r] & F_d\arrow[r] & Y \arrow[r]& 0
		\end{tikzcd}
		\]
		where the lower sequence is $d$-exact and can be assumed to be minimal by \cref{lem: minimal pushouts}. As $X_1$ is indecomposable, we have $F_i\in\uc$ for $1\leq i\leq d$. By \cite[Proposition 4.8 (ii)]{Jasso}, we have a  $d$-exact sequence
		\[
		0\to X\to X_1\oplus E_1\to F_1\oplus E_2\to \cdots \to F_{d-1}\oplus E_d\to F_d\to 0.
		\]
		Since $X\cong X_1\oplus X_2$, this complex can be written as the sum of the identity morphism 
		\[
		0\to X_1\xrightarrow{1_{X_1}}X_1\to 0 \to \cdots \to 0 \to 0\to 0
		\]
		and a $d$-exact sequence
		\begin{equation}\label{d-exactSeq1}
			0\to X_2\to E_1\to F_1\oplus E_2\to \cdots \to F_{d-1}\oplus E_d\to F_d\to 0.  
		\end{equation}
		Next we choose a decomposition $E_1\cong E_1''\oplus E_1'$ and $F_1\cong E_1''\oplus \Tilde{F_1}$ such that the morphism $E_1\to F_1$ becomes
		\begin{align*}
			\begin{pmatrix}
				1_{E_1''} & 0 \\ 0 & f
			\end{pmatrix}\colon E_1''\oplus E_1' \rightarrow E_1''\oplus \Tilde{F_1} 
		\end{align*}
		with $f \in \Rad_\ac(E_1', \Tilde{F_1})$. In particular, $E_1''$ is a summand of $F_1$, and is therefore contained in $\uc$. Hence, $E_1'\in \uc$ if and only if $E_1\in \uc$. Furthermore, \eqref{d-exactSeq1} can be written as a sum of the identity morphism
		\[
		0\to 0\to E_1''\xrightarrow{1_{E_1''}}E_1''\to 0 \to \cdots \to 0 \to 0\to 0
		\]
		and a $d$-exact sequence
		\begin{equation*}
			0\to X_2\to E'_1\to \Tilde{F_1}\oplus E_2\to \cdots \to F_{d-1}\oplus E_d\to F_d\to 0.  
		\end{equation*}
		By minimality of \eqref{Equation:OriginalExtension}, the morphism $E_1\to E_2$ is in the Jacobson radical, and hence the induced morphism $E_1'\to E_2$ is also in the Jacobson radical. Combining this with the fact that \mbox{$f\colon E_1'\to \Tilde{F_1}$} is in the Jacobson radical, we get that $E_1'\to \Tilde{F_1'}\oplus E_2$ is in the Jacobson radical. Hence, if we let
		\[
		0\to X_2\to E_1'\to F'_1\to \cdots \to F'_{d-1}\to F_d'\to 0
		\]
		be the minimal $d$-cokernel of $X_2\to E_1'$, then this must also give a minimal $d$-extension. By \Cref{prop: minimal}, the term $F'_d$ is a direct summand of $F_d$, and must therefore be in $\uc$. 
		Obviously, the object $X_2$ has fewer indecomposable summands than $X$. If $X_2$ is indecomposable, we know that $E'_1\in \uc$ (which is equivalent to $E_1\in \uc$), as the end terms $X_2$ and  $F_d'$ of the above minimal $d$-extension are in $\uc$. If not, we repeat the argument to eventually show that $E_1\in \uc$.
	\end{proof}
	
	We are now ready to give the proof of \cref{prop: extension closure indec}.
	
	\begin{proof}[Proof of \cref{prop: extension closure indec}]
		Suppose that we have a minimal $d$-extension
		\begin{align*}
			0 \to X \xrightarrow{e_0} E_1 \xrightarrow{e_1} E_2 \xrightarrow{e_2} \cdots \xrightarrow{e_{d-1}} E_d \xrightarrow{e_d} Y \to 0
		\end{align*}
		with $X, Y \in \uc$. By \cref{lem: indec first term}, we may assume that $X$ is indecomposable. Let $Y= \bigoplus_{j=1}^t Y_j$, where each $Y_j$ is indecomposable. For each inclusion $\iota_j \colon Y_j \to Y$, consider a $d$-pullback diagram
		\begin{equation}
			\begin{tikzcd}[column sep=22, row sep=22]
				0 \arrow[r] & X \arrow[r,"f_{j,0}"] \arrow[d,equal,"h_{j,0}"] & F_{j,1} \arrow[r,"f_{j,1}"] \arrow[d,"h_{j,1}"] & \cdots \arrow[r,"f_{j,d-1}"] & F_{j,d} \arrow[r,"f_{j,d}"] \arrow[d,"h_{j,d}"] & Y_j\arrow[r] \arrow[d, "\iota_j"] & 0 \\
				0 \arrow[r] & X \arrow[r,"e_0"] & E_1 \arrow[r,"e_1"] & \cdots  \arrow[r,"e_{d-1}"] & E_d \arrow[r,"e_d"] & Y \arrow[r] & 0. 
			\end{tikzcd} \label{eq 1}
		\end{equation}
		By the dual of \cref{lem: minimal pushouts}, the top $d$-extension can be chosen to be minimal. Consequently, each of the morphisms $f_{j,1}, \dots, f_{j,d-1}$ is in the Jacobson radical. Moreover, the middle objects $F_{j,1},\dots,F_{j,d}$ are in $\uc$ since $\uc$ is closed under $d$-extensions between indecomposables.
		
		Now look at the $d$-extension 
		\begin{align*}
			0 \to \bigoplus_{j=1}^t X \xrightarrow{f_0} \bigoplus_{j=1}^t F_{j,1} \xrightarrow{f_1}  \cdots \xrightarrow{f_{d-1}}  \bigoplus_{j=1}^t F_{j,d} \xrightarrow{f_d} \bigoplus_{j=1}^t Y_j  \to 0  
		\end{align*} 
		given by the direct sum of all the upper $d$-extensions obtained as in (\ref{eq 1}). Consider the induced map
		\begin{equation*}
			\begin{tikzcd}[scale=0.8,column sep=18, row sep=22]
				0 \arrow[r] & \bigoplus_{j=1}^t X \arrow[r,"f_0"] \arrow[d,"{h_0
				}"] & \bigoplus_{j=1}^t F_{j,1} \arrow[r,"f_1"] \arrow[d, "{h_1}"] & \cdots \arrow[r,"f_{d-1}"] & \bigoplus_{j=1}^t F_{j,d} \arrow[r,"f_d"] \arrow[d,"h_d"] & \bigoplus_{j=1}^t Y_j\arrow[r] \arrow[d, equal] & 0 
				\\
				0 \arrow[r] & X \arrow[r,"e_0"] & E_1 \arrow[r,"e_1"] & \cdots  \arrow[r,"e_{d-1}"] & E_d \arrow[r,"e_d"] & Y \arrow[r] & 0
			\end{tikzcd} \label{eq 3}
		\end{equation*}
		with $h_i =
		\begin{psmallmatrix}
			h_{1,i} & \cdots & h_{t,i}
		\end{psmallmatrix}$ for $i = 0,\dots,d$. This is a $d$-pushout diagram by \cite[Proposition 4.8]{Jasso}, and thus the associated mapping cone 
		\begin{align}
			0 \to \bigoplus_{j=1}^t X \xrightarrow{\begin{psmallmatrix}
					h_0 \\ -f_0
			\end{psmallmatrix}} X \oplus \bigoplus_{j=1}^t F_{j,1} \xrightarrow{\begin{psmallmatrix}
					e_0 & h_1 \\ 0 &  -f_1
			\end{psmallmatrix}} E_1 \oplus \bigoplus_{j=1}^t F_{j,2} \rightarrow \cdots \rightarrow E_d \to 0 \label{eq 2}
		\end{align}
		is a $d$-extension. Note that the term $X \oplus \bigoplus_{j=1}^t F_{j,1}$ lies in $\uc$. If (\ref{eq 2}) is given by the minimal $d$-cokernel of the first morphism, we are hence done by closure under minimal $d$-quotients. So suppose this $d$-cokernel is not minimal. By \cref{rem: minimal}, it is then isomorphic to the direct sum of the minimal $d$-cokernel and shifted complexes of the form $N \xrightarrow{1} N$. In particular, if 
		\begin{align*}
			X \oplus \bigoplus_{j=1}^t F_{j,1} \xrightarrow{\partial_1} M_1 \xrightarrow{\partial_2} M_2 \to \cdots \to M_d \to 0
		\end{align*}
		is the minimal $d$-cokernel of $\begin{psmallmatrix}
			h_0 \\ -f_0
		\end{psmallmatrix}$, then there is a commutative diagram 
		\[
		\begin{tikzcd}[column sep=30, row sep=25,ampersand replacement=\&]
			E_1 \oplus \bigoplus_{j=1}^t F_{j,2} \arrow[r, "{\begin{psmallmatrix} e_1 & h_2 \\ 0 & -f_2 \end{psmallmatrix}}"] \arrow[d,"{\begin{psmallmatrix} a & b \\ c & d \end{psmallmatrix}}"] \& E_2 \oplus \bigoplus_{j=1}^t F_{j,3}  \arrow[d, "\phi"]
			\\
			M_1 \oplus N_1 \arrow[r,"{\begin{psmallmatrix} \partial_2 & 0 \\ 0 & 1 \\ 0 & 0 \end{psmallmatrix}}"] \& M_2 \oplus N_1 \oplus N_2,
		\end{tikzcd}
		\]
		where the vertical maps are isomorphisms. Since $e_1 \in \Rad_\ac(E_1, E_2)$, the morphism
		\begin{align*}
			\phi \circ \begin{pmatrix} e_1
				\\ 0 \end{pmatrix}\colon E_1 \rightarrow M_2 \oplus N_1 \oplus N_2
		\end{align*}
		also lies in the radical, and thus so does
		\begin{align*}
			\phi \circ \begin{pmatrix} e_1 
				\\ 0 \end{pmatrix} =\begin{pmatrix} \partial_2 & 0 \\ 0 & 1 \\ 0 & 0 \end{pmatrix}  \begin{pmatrix} a \\ c \end{pmatrix}=  \begin{pmatrix} \partial_2\circ a \\ c \\ 0 \end{pmatrix}.  
		\end{align*}
		This shows that $c \in \Rad_\ac(E_1, N_1
		)$. Now let $ \begin{psmallmatrix} \alpha & \beta \\ \gamma & \delta \end{psmallmatrix}$ denote the inverse of $\begin{psmallmatrix} a & b \\ c & d \end{psmallmatrix}$. We have $\alpha a + \beta c =1_{E_1}$, or equivalently $\alpha a = 1_{E_1} - \beta c$. It follows from the definition of the radical that this is an isomorphism, as $c \in \Rad_\ac(E_1, N_1
		)$. This implies that $E_1$ is a direct summand of $M_1$. But $M_1 \in \uc$ since $\uc$ is closed under minimal $d$-quotients, and hence $E_1$ also lies in $\uc$. It then follows from \cref{lem:second term} that $E_2, \dots, E_d \in \uc $, so we can conclude that $\uc$ is closed under $d$-extensions as required.
	\end{proof}
	
	\section{The lattice of \texorpdfstring{$d$}{d}-torsion classes}\label{sec:lattice}
	
	The torsion classes in $\ac$ form a complete lattice with meet given by intersection, see e.g.\ \cite[Proposition 2.3]{IyamaReitenThomasTodorov}.
	In this section, we use the characterisation of higher torsion classes given in \cref{cor: characterisation} to show that an analogous statement holds for higher torsion classes.
	
	Let us first recall some relevant definitions.
	
	\begin{definition}
		Let $P$ be a poset. For an arbitrary subset $H \subseteq P$, the \textit{join} of $H$, if it exists, is the least upper bound of $H$. Dually, the \textit{meet} of $H$, if it exists, is the greatest lower bound of $H$. The poset $P$ is a \textit{complete lattice} if for any subset $H\subseteq P$, the join and the meet of $H$ exist.
	\end{definition}
	
	For the sake of clarity, note that a least upper bound is unique as it is smaller than any other upper bound, and similarly for greatest lower bounds. Note also that a complete lattice is bounded, i.e.\ it has a minimum and a maximum, obtained by letting $H$ in the definition be empty. The following lemma is well-known, see e.g.\ \cite[Chapter I, Lemma 34]{Graetzer}.
	
	\begin{lemma} \label{lem:lattice}
		Let $P$ be a poset. If every subset of $P$ admits a meet or if every subset of $P$ admits a join, then $P$ is a complete lattice.
	\end{lemma}
	
	For a $d$-cluster tilting subcategory $\mc$ of $\ac$, we let $\dtor(\mc)$ denote the poset of $d$-torsion classes in $\mc$ ordered by inclusion. 
	
	\begin{theorem}
		\label{theorem:lattice}
		Let $\mc$ be a $d$-cluster tilting subcategory of $\ac$. Then $\dtor(\mc)$ is a complete lattice with meet given by intersection.
	\end{theorem}
	
	\begin{proof} 
		To show that $\dtor(\mc)$ is a complete lattice, it suffices to show that it has arbitrary meets by \cref{lem:lattice}. We note that if $\dtor(\mc)$ is closed under arbitrary intersections, then meets are given by intersections, and closure under arbitrary meets follows. Therefore, we only need to show that for any subset $S$ of $\dtor(\mc)$, the intersection 
		\[
		\vc \colonequals \bigcap\limits_{\uc \in S} \uc
		\]
		is a $d$-torsion class.
		By \cref{cor: characterisation} and \cref{Lemma:ClosureMinimality}, it is enough to show that $\vc$ is closed under minimal $d$-extensions and minimal $d$-quotients. This follows from the fact that each $\uc\in S$ is closed under $d$-extensions and $d$-quotients by \cref{cor: characterisation}.
	\end{proof}
	
	We illustrate the lattice structure on the set of higher torsion classes in our running example.
	
	\begin{example}\label{ex:running4}
		In the setting of \cref{ex:running}, the set of all $2$-torsion classes in $\mc$ is listed in \cref{tab:intersection}. By \cref{theorem:lattice}, we know that the poset of all $2$-torsion classes in $\mc$ ordered by inclusion forms a complete lattice. We include the Hasse diagram in \cref{fig:latticeA3}.
		
		\begin{figure}[ht]
			\centering
			\[\begin{tikzcd}[column sep=1em, row sep=1em]
				& \mc & \\
				& &  \\
				& & \add \left\{ \rep{2\\3} \oplus \rep{1\\2} \oplus \rep{1}\right\} \\
				\\
				\add \left\{ \rep{3}\right\} & & \add \left\{\rep{1\\2} \oplus \rep{1}\right\} \\
				\\
				& & \add \left\{\rep{1}\right\} \\
				\\
				& \{0\} &
				\arrow[from=1-2, to=5-1]
				\arrow[from=1-2, to=3-3]
				\arrow[from=3-3, to=5-3]
				\arrow[from=5-3, to=7-3]
				\arrow[from=5-1, to=9-2]
				\arrow[from=7-3, to=9-2]
			\end{tikzcd}\]
			\caption{The Hasse diagram of the lattice of $2$-torsion classes in \cref{ex:running4}.}
			\label{fig:latticeA3}
		\end{figure}
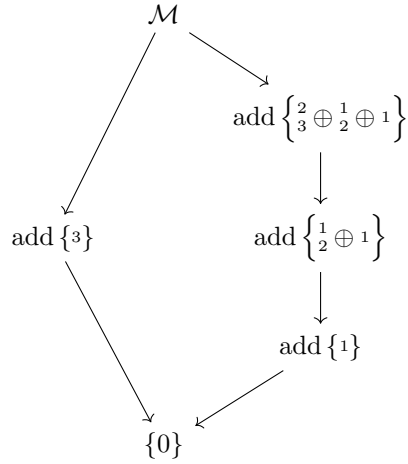
		
	\end{example}
	
	Another example, namely the complete lattice of $3$-torsion classes for the higher Auslander algebra $A_3^3$, can be found in \cref{Example:A33}. This example will demonstrate that, unlike in the classical setting, the lattice of $d$-torsion classes may not be Hasse-regular or semi-distributive (see e.g.\ \cite{DIRRT} for the definitions).
	
	It follows from \cref{thm:torsion-to-dtorsion} that there is an injective, order-preserving map 
	\begin{equation}\mathrm{T}(-):\dtor(\mc) \to \tor(\ac)  \label{non map of lattices}
	\end{equation}
	which takes a $d$-torsion class $\uc \subseteq \mc$ to the smallest torsion class in $\ac$ containing $\uc$ (see \cite[Corollary 3.3]{AJST}). We now give an example that demonstrates that this map is not a morphism of lattices.
	
	\begin{example}\label{ex:running5}
		Continuing with \cref{ex:running4}, we see that the $2$-torsion classes $\add\{3\}$ and $\add\{1\}$ are both torsion classes in $\ac$, and thus are sent to themselves under the map \eqref{non map of lattices}. 
		However, \cref{fig:latticeA3} shows that the join in $2$-$\tor(\mc)$ is $\mc$, which is sent to $\ac$ under \eqref{non map of lattices}, while the join in $\tor(\ac)$ is simply $\add\{3 \oplus 1\}$, which is contained in $\mc$ but it is not a $2$-torsion class.
	\end{example}
	
	\begin{remark}
		Note that for $\uc\in \dtor(\mc)$, the set
		\[
		\tor_{\,\uc}(\ac)\colonequals\{\tc\in \tor(\ac)\mid \tc\cap \mc=\uc\}
		\]
		may contain more than one element, is convex, and has a minimal element. Indeed, for the $2$-torsion class $\uc\colonequals\add \{\rep{1\\2} \oplus 1\}$ in \cref{ex:running5}, we have
		\[
		\add \{\rep{1\\2} \oplus 1\}\in \tor_{\,\uc}(\ac) \quad \text{and} \quad \add \{2 \oplus \rep{1\\2} \oplus 1\}\in \tor_{\,\uc} (\ac), 
		\]
		and hence $\tor_{\,\uc}(\ac)$ has more than one element. Also, for $\uc\in \dtor(\mc)$, if we have inclusions $\tc_1\subseteq \tc\subseteq \tc_2$ of torsion classes such that $\tc_1,\tc_2\in \tor_{\,\uc}(\ac)$, then
		$$\uc = \tc_1 \cap \mc \subseteq \tc \cap \mc \subseteq \tc_2 \cap \mc =\uc.$$  This shows that $\tc\in \tor_{\,\uc}(\ac)$, and hence $\tor_{\,\uc} (\ac)$ is convex.  
		Finally, $T(\uc)$ is the minimal element in $\tor_{\,\uc}(\ac)$ by \cref{thm:torsion-to-dtorsion}, where $\mathrm{T}(-)$ is as in \eqref{non map of lattices}. 
	\end{remark}

	We finish by investigating when the intersection with a $d$-cluster tilting subcategory gives a map of posets. 
	The result will be applied in the context of higher Nakayama algebras in \cref{sec:Higher NA}.
	
	\begin{proposition}\label{prop:IntersectionMapOfPoset}
		Let $\ac_1$ and $\ac_2$ be abelian categories of finite length, and let $\mc_1\subseteq \ac_1$ and $\mc_2\subseteq \ac_2$ be  $d$-cluster tilting subcategories. Assume we have an exact inclusion $\ac_2 \subseteq \ac_1$ such that $\mc_2\subseteq \mc_1$. The following statements hold:
		\begin{enumerate}
			\item\label{prop:IntersectionMapOfPoset:1} If $\uc$ is a $d$-torsion class in $\mc_1$, then $\uc\cap \mc_2$ is a $d$-torsion class in $\mc_2$.
			\item\label{prop:IntersectionMapOfPoset:2} Intersecting with $\mc_2$ gives a map of posets
			\[
			\dtor(\mc_1)\to \dtor(\mc_2)
			\]
			which preserves meets.
			\item\label{prop:IntersectionMapOfPoset:3} 
			If $\ac_2$ is closed under quotients in $\ac_1$, then $\mc_2$ is closed under $d$-quotients in $\mc_1$.
		\end{enumerate}
	\end{proposition}
	
	\begin{proof}
		Assume that $\uc$ is a $d$-torsion class in $\mc_1$. Since the inclusion $\ac_2\subseteq\ac_1$ is exact, it must send $d$-cokernels and $d$-kernels in $\mc_2$ to $d$-cokernels and $d$-kernels in $\mc_1$, respectively. In particular, it preserves $d$-quotients and $d$-extensions. Therefore, the subcategory $\uc\cap \mc_2$ must be closed under $d$-extensions and $d$-quotients in $\mc_2$, since $\uc$ is closed under $d$-extensions and $d$-quotients in $\mc_1$. This proves \eqref{prop:IntersectionMapOfPoset:1}. 
		
		Part \eqref{prop:IntersectionMapOfPoset:2} follows from \eqref{prop:IntersectionMapOfPoset:1} and the fact that meets in $\dtor(\mc_1)$ and $\dtor(\mc_2)$ are given by intersection. 
		
		For part \eqref{prop:IntersectionMapOfPoset:3}, note that giving a $d$-quotient in $\mc_1$ of an object $Y\in \mc_2$ is the same as giving a $d$-cokernel in $\mc_1$ of a morphism $X\to Y$ with $X\in \mc_1$. This is again equivalent to giving an exact sequence
		\[
		0\to C\to M_1\to M_2\to \cdots \to M_d\to 0
		\]
		where $C$ is the cokernel of $X\to Y$ and each $M_i$ is in $\mc_1$. Now since $\ac_2$ is closed under quotients in $\ac_1$, the cokernel $C$ must be contained in $\ac_2$. As $\mc_2$ is $d$-cluster tilting in $\ac_2$, we can construct an exact sequence
		\[
		0\to C\to N_1\to N_2\to \cdots \to N_d\to 0
		\]
		where each $N_i$ is in $\mc_2$, see \cite[Proposition 3.17]{Jasso}. 
		Since $\mc_2\subseteq \mc_1$, this must give a $d$-cokernel of $X\to Y$ in $\mc_1$ by the observation above. This proves the claim. 
	\end{proof}
	
	\begin{remark}\label{rem: d-quotients inherited}
		Note that if $\ac_2$ is closed under quotients in $\ac_1$, it follows from \cref{prop:IntersectionMapOfPoset}\eqref{prop:IntersectionMapOfPoset:3} that any minimal $d$-cokernel in $\mc_1$ of a morphism $X \to Y$ with $Y \in \mc_2$ is a minimal $d$-cokernel of a morphism in $\mc_2$.
		In particular, a subcategory of $\mc_2$ is closed under $d$-quotients in $\mc_2$ if and only if it is closed under $d$-quotients in $\mc_1$.
	\end{remark}
	
	%%%%%%%%%%%%%%%%%%%%%%%%%%%%%%
	
	\section{Higher Auslander algebras }\label{sec: higher AA}
	In this section, we apply \cref{cor: characterisation} to classify and count the $d$-torsion classes associated to higher Auslander algebras of type $\mathbb A$. Higher Auslander algebras were introduced in \cite{Iyama2011} and constitute an important class of algebras in higher homological algebra. The module category of each such algebra contains a $d$-cluster tilting subcategory, which was described combinatorially in \cite{Iyama2011} and \cite{OppermannThomas}.
	
	Recall from \cref{cor: characterisation} that $d$-torsion classes in a $d$-cluster tilting subcategory are precisely the subcategories which are closed under $d$-extensions and $d$-quotients. In \cref{Background on higher Auslander algebras,sec:closure d-quotients} we present results on closure under $d$-extensions and $d$-quotients for higher Auslander algebras of type $\mathbb A$, culiminating in a combinatorial characterisation of their higher torsion classes in \cref{thm:higherAuslanderalg}. In \cref{sec:computational results} we employ our results to write an algorithm which computes and counts all these $d$-torsion classes.
	
	\subsection{Background on higher Auslander algebras}\label{Background on higher Auslander algebras} 
	
	We start by providing a brief introduction to higher Auslander algebras, highlighting combinatorial descriptions which will be important throughout \cref{sec: higher AA}. We mostly follow the notation and terminology from \cite{Higher Nakayama}. 
	
	For positive integers $n$ and $d$, let $N_n=\{0,1, \dots, n-1\}$ with the natural poset structure. Consider the set
	\[
	N_{n}^d = \underbrace{N_n\times\cdots\times N_n}_{d\text{ times}}
	\]
	of $d$-tuples $x=(x_0,\dots,x_{d-1})$ over $N_n$. We endow $N_n^d$
	with the product order, meaning that $x\leq y$ in $N_{n}^d$ if and only if $x_i\leq y_i$ for all $i=0, 1, \ldots, d-1$. We consider $N_n^d$ as a category whose objects are the elements of  $N_n^d$, and whose morphisms are given by the poset relations of  $N_n^d$. Taking the $\kk$-linearisation of this category, we get a finite-dimensional $\kk$-algebra, see \cite[Section 1.2]{Higher Nakayama} for more details. By abuse of notation, we also denote this algebra by $N_n^d$.
	
	Let $\os_n^d$ be the subset of $N_n^d$ of non-decreasing $d$-tuples over $N_{n}$.  In particular, an element of $\os_n^d$ is a tuple $x=(x_0,\dots,x_{d-1})$ with $x_{0}\leq x_{1}\leq\dots\leq x_{d-1}$. The \textit{higher Auslander algebra} $A_n^{d}$ is defined as the idempotent quotient
	\[
	A_n^{d} \colonequals N_n^d/(N_n^d\setminus \os_n^d),
	\]
	where we consider $N_n^d$ as a finite-dimensional $\kk$-algebra as above. Note that $A_n^{d}$ is equivalently given by the opposite of
	a quiver $Q^{n,d}$ whose vertices are the elements of the set $\os_n^d$, and where there is an arrow from vertex $x$ to vertex $y$ if we have $y_i=x_i+1$ for exactly one $0\leq i\leq d-1$ and $y_j=x_j$ for $j\neq i$. The relations of $A_n^{d}$ are given by an admissible ideal $I_{n,d}$ making squares commutative and sending certain compositions of two arrows to zero, see \cite{HerschendJorgensen}. 
	
	\begin{remark}\label{rem:different notation}
		The notation we use is similar to that in \cite{Higher Nakayama}. It relates to the notation in \cite{HerschendJorgensen} in the following way: What we call $Q^{n,d}$, $A_n^d$ and $I_{n,d}$ corresponds to what is denoted by $Q^{n,d-1}$, $A_n^{d-1}$ and $I_{n,d-1}$ in \cite{HerschendJorgensen}. To see this, note that the poset $\os_n^d$ is isomorphic to the poset $\vc_{n,d-1}$ of \textit{increasing} $d$-tuples $x'=(x'_0,\dots,x'_{d-1})$ over $\{1,2,\dots ,n+d-1\}$ used in \cite{HerschendJorgensen}. The isomorphism is given by
		\begin{align*}
			\os_n^d\to \vc_{n,d-1} \quad (x_0,\dots ,x_{d-1})\mapsto (x_0+1,x_1+2,\dots, x_{d-1}+d). 
		\end{align*}
	\end{remark}
	
	The module category of $A_n^d$ has a $d$-cluster tilting subcategory
	\[
	\mc^d_{n} \colonequals \add(M^d_{n}) \subseteq \mod A^d_n
	\]
	where $M^d_{n}=\bigoplus_{x\in \os_n^{d+1}} M_x$. Here, the notation $M_x$ is used for the indecomposable $A_n^d$-module with support in all vertices $y\in \os_n^d$ such that $x_{0}\leq y_{0} \leq x_{1} \leq \cdots \leq x_{d-1} \leq y_{d-1} \leq x_{d} $. Note that the $d$-cluster tilting subcategory $\mc_n^d$ contains finitely many indecomposable objects, indexed by $\os_n^{d+1}$. It is known that $\End_{A_n^d}(M^d_{n})$ and $A^{d+1}_{n}$ are isomorphic as algebras by \cite[Corollary 1.16]{Iyama2011}, see also \cite[Theorem 2.3]{Higher Nakayama}.
	For examples of the quivers $Q^{n,d}$ and relevant modules, see \cite[Section 2.1]{Higher Nakayama} and \cite[Example 2.13]{HerschendJorgensen}. 
	
	Next we define the relation
	\begin{align*}
		x\rightsquigarrow y &\text{ if and only if } x_{0}\leq y_{0} \leq x_{1}\leq y_{1} \leq \cdots  \leq  x_{d}\leq y_{d}
	\end{align*}
	on the set $\mathbb{Z}^{d+1}$ of all $(d+1)$-tuples over $\mathbb{Z}$. Note that $x\rightsquigarrow y$ implies $x\leq y$. Using the relation $\rightsquigarrow$, one can determine the $\Hom$-spaces between indecomposable modules in $\mc^d_{n}$.
	
	\begin{proposition}{\cite[Theorem 3.6(3)]{OppermannThomas}, \cite[Proposition 2.8]{Higher Nakayama}}  \label{prop: morphism}
		Let $x,y\in \os_n^{d+1}$. Then \[\dim \Hom_{A_n^d}(M_x,M_y)=\begin{cases}
			1 & \text{if } x\rightsquigarrow y \\ 0 &\text{otherwise}.
		\end{cases}\]
	\end{proposition}
	
	The $d$-extensions in $\mc^d_{n}$ have a similar combinatorial description, including a description of all the middle terms. To this end, we define $\tau_d\colon \mathbb{Z}^{d+1}\to \mathbb{Z}^{d+1}$ by
	\begin{align*}
		\tau_d(x_0,\dots,x_{d})=(x_0-1,x_1-1,\dots, x_{d}-1).
	\end{align*}
	The notation is motivated by the fact that if $x\in \os_n^{d+1}$ with $x_0>0$, then 
	\(
	\tau_d(M_x)\cong M_{\tau_d(x)}
	\)
	by \cite[Proposition 2.7(iii)]{Higher Nakayama}, where $\tau_d (M_x)$ 
	is the higher Auslander--Reiten translate of $M_x$.
	\begin{proposition}{\cite[Theorem 3.6(4) and 3.8]{OppermannThomas}, \cite[Proposition 2.8]{Higher Nakayama}} \label{prop: ext descrip}
		Let $x,y\in \os_n^{d+1}$. Then  
		\[\dim \Ext^{d}_{A_n^d}(M_y,M_x)= \begin{cases}
			1 & \text{if } x \rightsquigarrow\tau_d(y) \\
			0 &\text{otherwise}.
		\end{cases}\]
		In particular, if $x\rightsquigarrow \tau_d(y)$, there is a non-trivial $d$-extension  \begin{align}\label{d-extension}
			0\rightarrow M_x \rightarrow E_1\rightarrow \cdots \rightarrow E_d \rightarrow M_y \rightarrow 0
		\end{align}
		where $E_k=\bigoplus_{z\in Z_k}M_z$ for 
		\[Z_k=\{z\in \os_n^{d+1} \mid z_i\in\{x_i,y_i\} \text{ for each } i \text{ and } | \{i \mid z_i=y_i \}|=k \}.\]  
	\end{proposition}
	
	\begin{remark}\label{rem: min d-ext}
		The $d$-extension (\ref{d-extension}) in \cref{prop: ext descrip} is minimal. This is seen by combining the fact that $E_k$ and $E_{k+1}$ have no isomorphic direct summands for $k = 1,\dots,d-1$ with \cref{prop: minimal} and \cref{rem: minimal}.
	\end{remark}
	
	We will use the following lemma, which can be seen as an immediate consequence of the description of $d$-extensions in \cref{prop: ext descrip}.
	
	\begin{lemma}  \label{prop: monomorphism}
		Suppose that $x,y\in \os_n^{d+1}$ with $x_i=y_i$ for all $i=0,\dots,d-1$ and $x_d \leq y_d$. Then any non-zero morphism $M_{x} \to M_y$ is a monomorphism.
	\end{lemma}
	
	\begin{proof}
		If $x_d=y_d$, the only non-zero morphism (up to multiplication by a scalar) is the identity, which is a monomorphism. If $x_d < y_d$, the  result follows from \cref{prop: ext descrip} when looking at the extension between $M_x$ and $M_z$, where $z=(x_1+1, \dots, x_d+1, y_d)$.
	\end{proof}
	
	Although \cref{prop: ext descrip} is limited to describing the middle terms in $d$-extensions with indecomposable end terms, we know from \cref{prop: extension closure indec} that this knowledge is sufficient for producing $d$-torsion classes if we have already established closure under $d$-quotients. 
	
	\subsection{A combinatorial characterisation of \texorpdfstring{$d$}{d}-torsion classes in \texorpdfstring{$\mc^{d}_{n}$}{Mnd}.}
	\label{sec:closure d-quotients}
	The main result of this subsection is a characterisation of higher torsion classes associated to higher Auslander algebras of type $\mathbb A$, see \cref{thm:higherAuslanderalg}. The key ingredient in the proof of this result is a combinatorial description of how to close a subcategory of $\mc^d_{n}$ under $d$-quotients.
	
	Given $M \in \mc^d_{n}$, we write $\dq(M)$ for the smallest subcategory of $\mc^d_{n}$ which contains $M$ and is closed under $d$-quotients (see \cref{def: dq closure}). We often refer to $\dq(M)$ as the \textit{$d$-quotient closure} of $M$. It is clear that if $N \in \dq(M)$, then $\dq(N) \subseteq \dq(M)$.
	
	We filter $\dq(M)$ as follows:
	\begin{itemize}
		\item Set $\dq(M)_0=\add(M)$.
		\item For $i \geq 0$, set 
		\[
		\dq(M)_{i+1} = \add\left\{ N \in \mc^d_{n} \ \middle\vert \begin{array}{c}
			\exists \ \text{minimal $d$-quotient} \ X \to Y \to C_1 \to \dots \to C_d \to 0 \ \text{in} \\ \mc^d_{n}
			\ \text{with} \ Y \in \dq(M)_i \ \text{and} \ N \cong C_j \ \text{for some $1\leq j \leq d$}
		\end{array}\right\}.
		\]
	\end{itemize} 
	We see that $\dq(M)_0 \subseteq \dq(M)_1 \subseteq \dots$ and that the chain must stabilise with $\dq(M)_t=\dq(M)_{t+1}$ for some $t \in \mathbb{N}$ because $\mc^d_{n}$ has finitely many indecomposables. By definition, we have \mbox{$\dq(M)=\dq(M)_t$}. 
	
	To completely determine the subcategory $\dq(M)$, it is sufficient to describe the indecomposable modules it contains. With this in mind, we begin by identifying certain indecomposables which must be contained in $d$-quotient closures.
	
	\begin{lemma} \label{lem: add 1}
		Let $x \in \os_n^{d+1}$ be such that $x_i +1 \leq x_{i+1}$ for some $0\leq i \leq d-1$ and set $y=(x_0, \dots, x_{i-1}, x_i +1, x_{i+1}, \dots, x_d)$. Then $M_{y} \in \dq(M_{x})$.
	\end{lemma}
	\begin{proof}
		Define $z=(z_0, \dots, z_d)$ such that
		\[
		z_j = \begin{cases} x_j & \text{if } j \neq  i+1 \\
			x_i & \text{if } j=i+1.
		\end{cases}
		\]
		We then have $z \rightsquigarrow x$, so there is a non-zero morphism $M_{z} \to M_x$ by \cref{prop: morphism}. 
		Because $x_i\leq z_{i+1}<x_{i+1}$, the module $M_{y}$ is hence in $\dq(M_{x})$ by \cite[Lemma 3.8(2)]{HerschendJorgensen}, keeping in mind that this paper uses a different notation as we outlined in \cref{rem:different notation}.
	\end{proof}
	
	\cref{lem: add 1} yields the following corollary.
	
	\begin{corollary} \label{cor: dq inclusion 1}
		Given any $x \in \os_n^{d+1}$, the set 
		\[
		\left\{ M_{y} \in \mc_n^d \mid y \in \os_n^{d+1}, \ x \leq y \ \text{and} \ x_d=y_d \right\}
		\]
		is contained in $\dq(M_{x})$.
	\end{corollary}
	\begin{proof}
		Suppose $y \in \os_n^{d+1}$ satisfies  $x \leq y$ and $x_d=y_d$. Construct a sequence \[x=z^0, z^1, z^2, \dots, z^m=y\] in $\os_n^{d+1}$, where the element $z^{i+1}$ is constructed from $z^i$ as follows. If $z^i=y$, then we are finished. Otherwise, there exists a maximal $j$ such that $z_j^i < y_j$, and we must have $j<d$. Then \mbox{$z_j^i+1 \leq y_j \leq y_{j+1}=z_{j+1}^i$} and we define 
		\[
		z^{i+1}_k = \begin{cases} z^i_k+1 & \text{if } k=j \\
			z^i_k & \text{if } k \neq j.
		\end{cases}
		\]
		Notice that $M_{z^{i+1}} \in \dq(M_{z^i})$ by \cref{lem: add 1} and that this process must terminate with $z^m=y$. Thus, we get 
		\[
		M_{y} \in \dq(M_{z^{m-1}}) \subseteq \dots \subseteq \dq(M_{x})
		\]
		as required.
	\end{proof}
	
	We now present the key technical lemma needed to describe $\dq(M_x)$ completely.
	
	\begin{lemma} \label{lem: lower bound 1}
		Suppose that $C_0 \to C_1 \to \dots \to C_d \to 0$ is a minimal $d$-cokernel of a morphism $C_{-1} \rightarrow C_0$ in $\mc^d_{n}$. If $M_{z} \in \add(C_i)$ for some $1\leq i \leq d$, then there exists $M_{x} \in \add(C_0)$ such that $x\leq z$ and $x_d=z_d$. 
	\end{lemma}
	
	\begin{proof}
		Since each $C_i \in \mc_n^d$, we may assume that every $C_i$ is equal to a direct sum of modules of the form $M_z$ for $z\in\os_n^{d+1}$.
		Choose $M_{z} \in \add(C_i)$. Write the morphism $C_{i-1} \to C_i$ as 
		\[
		\begin{pmatrix} f \\ g \end{pmatrix}\colon C_{i-1} \rightarrow M_{z} \oplus C_i'.
		\]
		By the construction of minimal $d$-cokernels, this morphism factors through the cokernel $K$ of the morphism $C_{i-2} \to C_{i-1}$ 
		as indicated in the diagram
		\[
		\begin{tikzcd}
			C_{i-1} \arrow[rr,"\begin{psmallmatrix} f \\ g \end{psmallmatrix}"]\arrow[rd, two heads, "\pi"] &  & M_{z} \oplus C_i',\\
			& K \arrow[ru, swap, "\begin{psmallmatrix} f' \\ g' \end{psmallmatrix}"]
		\end{tikzcd}
		\]
		where $\begin{psmallmatrix} f' \\ g' \end{psmallmatrix}$ is a minimal left $\mc^d_{n}$-approximation of $K$.
		
		There exists some $M_{w} \in \add(C_{i-1})$ with a non-zero morphism $M_{w} \to M_{z}$. Indeed, if this is not the case, then $f=0$. However, since $\pi$ is an epimorphism, this would imply $f'=0$, contradicting $\begin{psmallmatrix} f' \\ g' \end{psmallmatrix}$ being left minimal. Consequently, we may write $C_{i-1} = D \oplus D'$, where $D \neq 0$ and $D'$ is the largest summand of $C_{i-1}$ that maps to zero under $f$. This means that the morphism $\begin{psmallmatrix} f \\ g \end{psmallmatrix}$ may be written as 
		\begin{align*}
			\begin{pmatrix} f_1 & 0 \\ g_1 & g_2 \end{pmatrix}\colon
			D \oplus D' \xrightarrow{} M_z \oplus C_i',
		\end{align*}
		where $f_1$ is nonzero.
		For every $M_{v} \in \add(D)$, there is a non-zero morphism to $M_z$. Hence, \cref{prop: morphism} shows that $v \rightsquigarrow z$. 
		Since $D \neq 0$, we may choose $M_y \in \add(D)$ with $y_d$ maximal, i.e.\ for all $M_{v} \in \add(D)$, we have $v_d \leq y_d.$
		
		With $z$ and $y$ fixed and knowing that $y \rightsquigarrow z$, we may now consider $z'=(z_0,z_1, \dots, z_{d-1},y_d)$. Notice that $z' \rightsquigarrow z$. \cref{prop: morphism} and \cref{prop: monomorphism} thus yield that there exists a monomorphism $\iota \colon M_{z'} \to M_z$. Moreover, by the maximality of $y_d$, we have $v \rightsquigarrow z'$ for all $M_{v} \in \add(D)$. Hence, by \cref{prop: morphism}, there exists a morphism $h \colon D \rightarrow M_{z'}$ such that $f_1 = \iota \circ h$.
		In particular, the solid part of the diagram 
		\[
		\begin{tikzcd}[ampersand replacement=\&]
			D \oplus D' \arrow{r}{
				\begin{psmallmatrix} h & 0 \\ g_1 & g_2 \end{psmallmatrix}}
			\arrow[d, two heads, swap, "\pi=\begin{psmallmatrix} \pi_1 \\ \pi_2 \end{psmallmatrix}"
			] \& M_{z'} \oplus C_i' \arrow[tail]{d}{ \begin{psmallmatrix} \iota & 0 \\ 0 & 1 \end{psmallmatrix}}
			\\
			K \arrow[ru,dashed,"t"] \arrow[r, tail, swap, "\begin{psmallmatrix} f' \\ g' \end{psmallmatrix}"
			] \& M_{z} \oplus C_i'
		\end{tikzcd}
		\]
		commutes. Since $\begin{psmallmatrix} \iota & 0 \\ 0 & 1 \end{psmallmatrix}$ is a monomorphism, this implies that the composition of the morphism $C_{i-2} \to C_{i-1}$
		and $\begin{psmallmatrix} h & 0 \\ g_1 & g_2 \end{psmallmatrix}$ is zero. Thus, there is a morphism $t \colon K \to M_{z'} \oplus C_i'$ making the upper triangle in the diagram above commute.
		
		Note that the lower triangle also commutes as $\pi$ is an epimorphism. Finally, since $\begin{psmallmatrix} f' \\ g' \end{psmallmatrix}$ is an $\mc^d_{n}$-approximation, there exists a morphism $s \colon M_z \oplus C_i' \to M_{z'} \oplus C_i'$ such that 
		\begin{align*}
			t=s \circ \begin{pmatrix} f' \\ g' \end{pmatrix}.
		\end{align*}
		It follows that
		\begin{align*}
			\begin{pmatrix} f' \\ g' \end{pmatrix} = 
			\begin{pmatrix} \iota & 0 \\ 0 & 1 \end{pmatrix} \circ t = \begin{pmatrix} \iota & 0 \\ 0 & 1 \end{pmatrix} \circ s \circ \begin{pmatrix} f' \\ g' \end{pmatrix},  
		\end{align*}
		and thus $\begin{psmallmatrix} \iota & 0 \\ 0 & 1 \end{psmallmatrix} \circ s$ is an isomorphism, since $\begin{psmallmatrix} f' \\ g' \end{psmallmatrix}$ is left minimal. Therefore, the monomorphism $\iota$ is also an epimorphism, so it must be an isomorphism and we have $M_{z}\cong M_{z'}$. This shows that $y_d=z_d$, and hence $M_y \in \add(C_{i-1})$ satisfies both $y \rightsquigarrow z$ (and thus $y \leq z$) and $y_d=z_d$. 
		
		We can now repeat the argument with $M_y$ and keep going until we get $M_{x} \in \add(C_0)$ with $x\leq y\leq z$ and $x_d=y_d=z_d$.
	\end{proof}
	
	\cref{lem: lower bound 1} enables us to fully describe $\dq(M_x)$ for an indecomposable module $M_x \in \mc^d_{n}$.
	
	\begin{corollary}\label{cor: describe dq}
		Given any $x \in \os_n^{d+1}$, we have
		\[
		\dq(M_x)=\add\left\{ M_{y} \in \mc_n^d \mid y \in \os_n^{d+1}, ~x \leq y \text{ and } x_d=y_d\right\}.
		\]
	\end{corollary}
	
	\begin{proof}
		By \cref{cor: dq inclusion 1}, it suffices to show that for any  $M_y \in \dq(M_x)$, we must have $x \leq y$ and $x_d=y_d$. Let $i\geq 0$ be such that $M_y \in \dq(M_x)_i$. If $i=0$, we have $M_y \in \dq(M_x)_0 = \add(M_x)$, so $M_y = M_x$ and the statement holds. Assume hence $i>0$. By \cref{lem: lower bound 1}, there exists $M_z \in \dq(M_x)_{i-1}$ such that $z \leq y$ and $z_d=y_d$. 
		
		Repeating this argument, we will eventually find some $M_w \in \dq(M_x)_0$ such that $w\leq y$ and $w_d=y_d$. As $\dq(M_x)_0=\add(M_x)$ yields $M_w=M_x$, the result follows.
	\end{proof}
	
	\begin{remark}\label{rem:d-tors generated by M_x}
		It follows from \cref{cor: describe dq} that any $M_y, M_z \in \dq(M_x)$ satisfy $y_d =x_d=z_d$, and thus there are no non-trivial $d$-extensions between them by \cref{prop: ext descrip}. In particular, the subcategory $\dq(M_x)$ is closed under $d$-extensions, and it is hence the smallest $d$-torsion class containing $M_x$ by \cref{cor: characterisation}.
	\end{remark}
	
	We now consider the $d$-quotient closure of a set of indecomposables.
	
	\begin{proposition}\label{prop: dq quotient}
		Given any subset $I \subseteq \os_n^{d+1}$, suppose that $M_y \in \dq\!\left(\bigoplus_{x \in I} M_x\right)$. Then $M_y \in \dq(M_x)$ for some $x \in I$. 
	\end{proposition}
	
	\begin{proof}
		We prove this by induction on the filtration of $\dq(\bigoplus_{x \in I} M_x)$. If 
		\[
		M_y \in \dq\!\left(\bigoplus_{x \in I} M_x\right)_{0}=\add\!\left(\bigoplus_{x \in I} M_x\right),
		\]
		the statement clearly holds.  
		
		Now suppose $M_y \in \dq\!\left(\bigoplus_{x \in I} M_x\right)_i$ for some $i>0$, and that the result is known for all $M_z \in \dq\!\left(\bigoplus_{x \in I} M_x\right)_{i-1}$. By construction, there must exist a minimal $d$-quotient
		\begin{align*}
			X \xrightarrow{f} Y \to C_1 \to \dots \to C_d \to 0
		\end{align*}
		in $\mc^d_{n}$ of some morphism $f$ such that $Y \in \dq\!\left(\bigoplus_{x \in I} M_x\right)_{i-1}$ and $M_y \in \add(C_j)$ for some $1\leq j \leq d$. Hence, combining \cref{lem: lower bound 1} with \cref{cor: dq inclusion 1} shows that there exists 
		\[
		M_z \in \add(Y) \subseteq \dq\!\left(\bigoplus_{x \in I} M_x\right)_{i-1}
		\]
		such that $M_y \in \dq(M_z)$. By the induction hypothesis, we have $M_z \in \dq(M_x)$ for some $x \in I$, and thus it follows that $M_y \in \dq(M_z) \subseteq \dq(M_x)$ as required.
	\end{proof}
	
	Using \cref{prop: dq quotient}, we are able to give a complete description of the $d$-quotient closure of any module in $\mc^d_{n}$. 
	
	\begin{corollary}\label{cor: dqOnSumOfIndec}
		Given a subset $I \subseteq \os_n^{d+1}$, the set of indecomposable modules in $\dq\!\left(\bigoplus_{x \in I} M_x\right)$ is precisely
		\begin{align*}
			\bigcup_{x \in I} \left\{ M_{y} \in \mc_n^d \mid y \in \os_n^{d+1}, \ x\leq y \text{ and }  x_d=y_d
			\right\}.
		\end{align*}
	\end{corollary}
	\begin{proof}
		Since it is clear that $\dq(M_x) \subseteq \dq\!\left(\bigoplus_{x \in I} M_x\right)$ for each $x \in I$, this is a direct consequence of \cref{cor: describe dq} and  \cref{prop: dq quotient}.
	\end{proof}
	
	For a set of indecomposable modules in $\mc^d_{n}$, \cref{cor: dqOnSumOfIndec} gives a purely combinatorial description of all the indecomposables in their $d$-quotient closure. Unlike what we saw in \cref{rem:d-tors generated by M_x}, there may now exist non-trivial $d$-extensions, so the $d$-quotient closure is not necessarily a $d$-torsion class. However, using the results developed in this subsection, we give a full combinatorial description of the subsets corresponding to $d$-torsion classes. Note that any subcategory of $\mc^d_{n}$ which is closed under direct summands is uniquely determined by its indecomposable modules, i.e.\ by a subset of $\os_n^{d+1}$. We hence use the notation
	\begin{align*}
		\uc_I \colonequals \add\{ M_{y} \in \mc_n^d \mid y \in I\}
	\end{align*}
	for the subcategory of $\mc^d_{n}$ associated to a subset $I \subseteq \os_n^{d+1}$.
	
	%%%%%%%%%%%%%%%%%%%%%%%%%%%%%%%%%%%
	
	\begin{theorem}\label{thm:higherAuslanderalg}
		Consider a subset $I \subseteq \os_n^{d+1}$. The subcategory $\uc_I$ is a $d$-torsion class in $\mc^d_{n}$ if and only if the following hold for any elements $x,z\in \os_n^{d+1}$:
		\begin{enumerate}
			\item \label{thm:higherAuslanderalg:1} If $x\leq z$ and $x_d=z_d$, then $x\in I$ implies $z\in I$.
			\item \label{thm:higherAuslanderalg:2} If $x\rightsquigarrow \tau_d(z)$ and $x,z\in I$, then any $y\in \os_n^{d+1}$ with
			$y_i \in \{x_i,z_i\}$ for each $i$ must be in $I$. 
		\end{enumerate}
	\end{theorem}
	
	\begin{remark}
		The product order on $N_n^{d+1}$ restricts to a partial order on the subset 
		\[
		\{y\in \os_n^{d+1} \mid y_d=m\} \subseteq N_n^{d+1}.
		\]
		\cref{thm:higherAuslanderalg}(\ref{thm:higherAuslanderalg:1}) is equivalent to $\{y\in I\mid y_d=m\}$ being an \textit{upper set} in $\{y\in \os_n^{d+1} \mid y_d=m\}$ for each $m= 0,\ldots, n-1$.
	\end{remark}
	
	\begin{proof}[Proof of \cref{thm:higherAuslanderalg}]
		By \cref{cor: dqOnSumOfIndec}, condition (\ref{thm:higherAuslanderalg:1}) holds if and only if $\uc_I$ is closed under $d$-quotients. It follows from \cref{prop: ext descrip} and \cref{rem: min d-ext} that condition (\ref{thm:higherAuslanderalg:2}) is equivalent to $\uc_I$ being closed under $d$-extensions by indecomposables. Moreover, if $\uc_I$ is closed under $d$-quotients, it is closed under $d$-extensions by indecomposables if and only if it is closed under all $d$-extensions by \cref{prop: extension closure indec}. From \cref{cor: characterisation}, we know that $\uc_I$ is a $d$-torsion class if and only if it is closed under $d$-extensions and $d$-quotients, which proves the claim.
	\end{proof}
	
	\subsection{Computational results}\label{sec:computational results}
	
	We are now ready to present two algorithms for explicitly computing higher torsion classes associated to higher Auslander algebras of type $\mathbb A$. The first algorithm computes the minimal $d$-torsion class containing a given module. The second computes all $d$-torsion classes associated to a higher Auslander algebra. Python code implementing the algorithms is available as a Google Colab notebook.\footnote{
		\url{https://colab.research.google.com/drive/172Q-UZHvdPOhngGkl1T_xdYLzntg31dY}}
	
	We extend our notation for the sake of readability. In particular, for an element $x\in \os_n^{d+1}$, we set $\dq(x)=\{y \in \os_n^{d+1} \mid M_y\in \dq(M_x)\}$. By \cref{cor: describe dq}, we know that
	\[
	\dq(x)=\left\{ y \in \os_n^{d+1} \mid x\leq y\ \text{and}\  x_d=y_d \right\}.
	\]
	Given a set of indecomposable modules $M_{x^1}, \dots, M_{x^r}$ in \mbox{$\mc^d_{n}$,} we let $\uc(M_{x^1}, \dots, M_{x^r})$ denote the smallest $d$-torsion class in $\mc^d_{n}$  containing $M_{x^1}, \dots, M_{x^r}$. For a set $X = \{x^1, \dots, x^r\} \subseteq \os_n^{d+1}$, we let $U(X)$ be the set of $(d+1)$-tuples in $\os^{d+1}_n$ corresponding to the indecomposables in $\uc(M_{x^1}, \dots, M_{x^r})$. If the set $X$ is empty, then $U(X)$ is also empty and corresponds to the trivial  $d$-torsion \mbox{class $\{0\}$}. 
	
	\begin{algorithm}\label{alg:make torsion class}
		Given a set of initial indecomposable modules in $\mc_n^d$, or equivalently a subset of $\os_n^{d+1}$, this algorithm computes the minimal $d$-torsion class containing those modules.
		
		Input: Integers $d\geq 1$, $n\geq 1$  and a set $X = \{x^1, \dots, x^r\} \subseteq \os_n^{d+1}$.
		\begin{enumerate}
			\item \label{alg:maketorsion:1}Let $I=X$. 
			\item \label{alg:maketorsion:2}For each pair $x,y\in I$ such that $x\rightsquigarrow \tau_d(y)$, add the $(d+1)$-tuple $(x_0, \dots,x_{d-1}, y_d)$ to $I$.
			\item \label{alg:maketorsion:3}For every $x\in I$, add the elements of $\dq(x)$ to $I$.
			\item \label{alg:maketorsion:4}If new elements were added to $I$ in step (\ref{alg:maketorsion:2}) or (\ref{alg:maketorsion:3}), repeat from step (\ref{alg:maketorsion:2}). Otherwise, terminate the process.
		\end{enumerate}
		
		Output: The set $I$.
	\end{algorithm}
	
	Since $\os_n^{d+1}$ is a finite set, \cref{alg:make torsion class} will always terminate and give a subset $I \subseteq \os_n^{d+1}$ as output. Recall that we use the notation
	\begin{align*}
		\uc_I \colonequals \add\{ M_{y} \in \mc_n^d \mid y \in I\}
	\end{align*}
	for the corresponding subcategory of $\mc^d_{n}$. \cref{lem: first algorithm} shows that the set $I$ produced in \cref{alg:make torsion class} indeed corresponds to the minimal $d$-torsion class containing the indecomposable modules we started with.
	
	\begin{proposition}\label{lem: first algorithm}
		The set $I$ constructed in \cref{alg:make torsion class} satisfies
		$\uc_I=\uc(M_{x^1}, \dots, M_{x^r})$.
	\end{proposition}
	
	\begin{proof}
		We need to show that $\uc_I$ is the minimal $d$-torsion class containing $M_{x^1}, \dots, M_{x^r}$. Step (\ref{alg:maketorsion:1}) of the algorithm ensures that $M_{x^1}, \dots, M_{x^r}\in\uc_I$. By \cref{prop: dq quotient}, step (\ref{alg:maketorsion:3}) implies that $\uc_I$ is closed under $d$-quotients.
		
		Consider two indecomposable modules $M_x, M_y \in \uc_I$ with $\Ext^d_{A_n^d}(M_y,M_x) \neq 0$. This means that the pair $x,y \in I$ satisfies $x\rightsquigarrow \tau_d(y)$ by \cref{prop: ext descrip}. Moreover, if $M_z$ is a direct summand in one of the middle terms in the non-trivial $d$-extension from $M_x$ to $M_y$ described in \cref{prop: ext descrip}, then either $z\in \dq(x)$ or $z\in \dq((x_{0}, \dots,x_{d-1}, y_d))$. Step (\ref{alg:maketorsion:2}) followed by step (\ref{alg:maketorsion:3}) thus ensures that $M_z \in \uc_I$, so $\uc_I$ is closed under  $d$-extensions with indecomposable end terms. By \cref{prop: extension closure indec}, this implies that $\uc_I$ is closed under all $d$-extensions. We can hence conclude that $\uc_I$ is a $d$-torsion class by \cref{cor: characterisation}.
		
		We now know that $\uc_I$ is a $d$-torsion class containing the modules $M_{x^1}, \dots, M_{x^r}$. 
		However, objects added to $I$ in the algorithm correspond to either $M_{x^1}, \dots, M_{x^r}$ or to indecomposable direct summands obtained from minimal $d$-quotients or minimal $d$-extensions, see \cref{rem: min d-ext}, and the result follows. 
	\end{proof}
	
	\begin{remark}\label{rem: impr alg 1}
		An improvement of \cref{alg:make torsion class}, omitted for the sake of readability, is that on subsequent iterations it suffices to only consider $(d+1)$-tuples added in the previous iteration in step (\ref{alg:maketorsion:3}). Similarly, in step (\ref{alg:maketorsion:2}), one only needs to to consider pairs where at least one $(d+1)$-tuple was added in the previous iteration.
	\end{remark}
	
	Building on \cref{alg:make torsion class}, we give an algorithm that determines all $d$-torsion classes in $\mc_n^d$.
	
	\begin{algorithm}\label{alg: find torsion classes}
		This algorithm computes all higher torsion classes associated to a higher Auslander algebra.
		
		Input: Integers $d\geq 1$, $n\geq 1$.
		\begin{enumerate}
			\item \label{alg:findtorsion:1} Let $\ucc$ be the singleton set containing the empty set. Set $l=1$.
			\item \label{alg:findtorsion:2} For all sets $X$ consisting of $l$ distinct $(d+1)$-tuples in $\os_n^{d+1}$, compute $U(X)$ using \cref{alg:make torsion class} and add it to $\ucc$.
			\item \label{alg:findtorsion:3} If new elements were added to $\ucc$ in step (\ref{alg:findtorsion:2}), increase $l$ by one and repeat from step (\ref{alg:findtorsion:2}). Otherwise, terminate the process.
		\end{enumerate}
		
		Output: The set $\ucc$.
	\end{algorithm}
	
	Similarly as in the case of \cref{alg:make torsion class}, it should be noted that the algorithm above must terminate as $\os_n^{d+1}$ is a finite set.
	
	\begin{proposition}
		The $d$-torsion classes in $\mc^d_{n}$ are indexed by $\ucc$. In other words,
		\[
		\dtor(\mc^d_{n})=\left\{\uc_I  \mid I\in\ucc \right\}.
		\]
	\end{proposition}
	
	\begin{proof}
		First observe that the trivial $d$-torsion class $\{0\}=\uc_{\emptyset}$, and $\emptyset\in\ucc$.
		As $\mc_n^d$ has finitely many indecomposable objects, any non-trivial $d$-torsion class $\vc$ in $\mc_n^d$ can be written as \mbox{$\vc=\uc(M_{x^1}, \dots, M_{x^m})$} for some positive integer $m$.
		
		Let $r$ be the lowest positive integer for which any $d$-torsion class of the form $\uc(M_{x^1}, \dots, M_{x^r})$ can also be written as $\uc(M_{y^1}, \dots, M_{y^s})$ for some $s<r$. By construction, the set $\ucc$ indexes all $d$-torsion classes of the form $\uc(M_{x^1}, \dots, M_{x^l})$ for $l<r$. 
		
		Suppose that $\vc$ is a non-trivial $d$-torsion class, and let $m\geq 0$ be minimal such that 
		\[
		\vc=\uc(M_{x^1}, \dots, M_{x^m}).
		\]
		We claim that $\vc=\uc_I$ for some $I\in\ucc$, i.e.\ that $m<r$. Indeed, if $m\geq r$, consider the $d$-torsion class $\vc'= \uc(M_{x^1}, \dots, M_{x^{r}})$. By the assumption on $r$, we can write $\vc'$ as $\uc(M_{z^1}, \dots, M_{z^s})$ for some $s< r$. But then $\vc$ is of the form $\uc(M_{z^1}, \dots M_{z^s},M_{x^{r+1}},\dots M_{x^m})$, and can hence be generated by $m+s-r<m$ indecomposable modules. This contradicts the minimality of $m$, and the result follows.
	\end{proof}
	
	\begin{remark}
		As before, \cref{alg: find torsion classes} is presented in its simplest form for the sake of readability. 
		For efficient computations, we use the following improvements: 
		\begin{itemize}
			\item We precalculate whether $x \rightsquigarrow y$ and  $x \rightsquigarrow \tau_d(y)$ for all $x,y \in \os_n^{d+1}$.
			\item If a set $X$ contains $x,y$ with $y\in \dq(x)$, then $U(X)$ has already been added to $\ucc$, so we skip the computation of $U(X)$ in step (\ref{alg:findtorsion:2}).
		\end{itemize}
	\end{remark}
	
	Using \cref{alg: find torsion classes}, we can compute the number of higher torsion classes associated to a higher Auslander algebra. These computational results are summarised in \cref{tab:Auslander torsion classes}.
	
	\begin{table}[thb]
		\centering
		\begin{tabular}{|cc|llllll|}\hline
			&&\multicolumn{6}{|c|}{$n$}\\
			& 	&1 & 2 & 3 & 4 & 5 & 6\\\hline
			\multirow{8}{1em}{$d$}	&1 & 2 & 5 & 14 & 42 & 132 & 429 \\
			&2 & 2 & 6 & 25 & 140 & 1036 & 10040\\
			&3 & 2 & 7 & 46 & 643 & 22224 & \\
			&4 & 2 & 8 & 87 & 4147 & & \\
			&5 & 2 & 9 & 168 & 36543 & & \\
			&6 & 2 & 10 & 329 & 427527 & & \\
			&7 & 2 & 11 & 650 & & & \\
			&8 & 2 & 12 & 1291 & & & \\\hline
		\end{tabular}
		\caption{The number of $d$-torsion classes in the $d$-cluster tilting subcategory $\mc^d_{n}$ of the higher Auslander algebra $A_{n}^d$.}  
		\label{tab:Auslander torsion classes}
	\end{table}
	
	In addition to computing the full set of $d$-torsion classes in $\mc_n^d$, our code also produces the associated Hasse diagram. Note that it gives a fully annotated version of the Hasse diagram, specifying the indecomposable modules contained in each $d$-torsion class.
	
	\begin{example}\label{Example:A33}
		Consider the higher Auslander algebra $A_{3}^{3}$. The Hasse diagram of the $3$-torsion classes in the $3$-cluster tilting subcategory $\mc_3^3$ is shown in \cref{fig:Hasse}.
		We note that the vertices labelled $w$, $x$ and $v$ have valency $3$, $4$ and $5$, respectively, so the lattice is not Hasse-regular. We moreover note that the lattice is not semi-distributive as it fails the criterion of meet-semi-distributivity \cite[p.\ 479]{Graetzer}. Indeed, using the notation $\vee$ for join and $\wedge$ for meet, we see that $x\wedge y=v=x\wedge z$, but that $x \wedge (y\vee z)=x \wedge w =x \neq v$.
	\end{example}
	
	\begin{figure}[pht]
		\centering
		
		\begin{tikzpicture}[line join=bevel, yscale=.5, xscale=1]
			\node (1) at (127.0bp,810.0bp) {$\bullet$};
			\node (23) at (199.0bp,738.0bp) {$\bullet$};
			\node (25) at (127.0bp,738.0bp) {$\bullet$};
			\node (2) at (55.0bp,810.0bp) {$\bullet$};
			\node (35) at (55.0bp,594.0bp) {$\bullet$};
			\node (26) at (127.0bp,954.0bp) {$\bullet$};
			\node (36) at (190.0bp,882.0bp) {$\bullet$};
			\node (38) at (127.0bp,882.0bp) {$\bullet$};
			\node (22) at (199.0bp,666.0bp) {$\bullet$};
			\node (33) at (271.0bp,666.0bp) {$\bullet$};
			\node (5) at (170.0bp,234.0bp) {$\bullet$};
			\node (4) at (190.0bp,162.0bp) {$\bullet$};
			\node (16) at (199.0bp,594.0bp) {$\bullet$};
			\node (9) at (254.0bp,522.0bp) {$\bullet$};
			\node (15) at (110.0bp,522.0bp) {$\bullet$};
			\node (30) at (326.0bp,378.0bp) {$z$};
			\node (28) at (271.0bp,306.0bp) {$\bullet$};
			\node (29) at (326.0bp,306.0bp) {$\bullet$};
			\node (11) at (215.0bp,306.0bp) {$v$};
			\node (27) at (271.0bp,234.0bp) {$\bullet$};
			\node (43) at (262.0bp,162.0bp) {$\bullet$};
			\node (32) at (326.0bp,522.0bp) {$\bullet$};
			\node (31) at (326.0bp,450.0bp) {$\bullet$};
			\node (19) at (254.0bp,450.0bp) {$\bullet$};
			\node (24) at (127.0bp,666.0bp) {$\bullet$};
			\node (7) at (254.0bp,378.0bp) {$x$};
			\node (3) at (190.0bp,90.0bp) {$\bullet$};
			\node (13) at (182.0bp,450.0bp) {$\bullet$};
			\node (34) at (271.0bp,738.0bp) {$\bullet$};
			\node (14) at (110.0bp,450.0bp) {$\bullet$};
			\node (6) at (215.0bp,234.0bp) {$\bullet$};
			\node (40) at (199.0bp,810.0bp) {$\bullet$};
			\node (37) at (55.0bp,1026.0bp) {$\bullet$};
			\node (39) at (55.0bp,954.0bp) {$\bullet$};
			\node (42) at (110.0bp,90.0bp) {$\bullet$};
			\node (10) at (110.0bp,378.0bp) {$\bullet$};
			\node (8) at (130.0bp,306.0bp) {$\bullet$};
			\node (21) at (55.0bp,306.0bp) {$\bullet$};
			\node (41) at (127.0bp,1026.0bp) {$\bullet$};
			\node (45) at (127.0bp,1098.0bp) {$\mc^3_{3}$};
			\node (44) at (271.0bp,810.0bp) {$w$};
			\node (18) at (127.0bp,594.0bp) {$\bullet$};
			\node (17) at (182.0bp,522.0bp) {$\bullet$};
			\node (20) at (271.0bp,594.0bp) {$\bullet$};
			\node (0) at (110.0bp,18.0bp) {$\{0\}$};
			\node (12) at (182.0bp,378.0bp) {$y$};
			\draw [->] (1) ..controls (152.2bp,784.8bp) and (165.32bp,771.68bp)  .. (23);
			\draw [->] (1) ..controls (127.0bp,784.13bp) and (127.0bp,774.97bp)  .. (25);
			\draw [->] (2) ..controls (55.0bp,754.39bp) and (55.0bp,667.55bp)  .. (35);
			\draw [->] (26) ..controls (149.17bp,928.66bp) and (160.26bp,915.99bp)  .. (36);
			\draw [->] (26) --(38);% ..controls (123.77bp,928.13bp) and (122.62bp,918.97bp)  .. (38);
			\draw [->] (23) ..controls (199.0bp,712.13bp) and (199.0bp,702.97bp)  .. (22);
			\draw [->] (23) ..controls (224.2bp,712.8bp) and (237.32bp,699.68bp)  .. (33);
			\draw [->] (5) --(4);%..controls (149.17bp,208.66bp) and (160.26bp,195.99bp)  .. (4);
			\draw [->] (16) ..controls (218.51bp,568.46bp) and (227.43bp,556.78bp)  .. (9);
			\draw [->] (16) ..controls (168.17bp,569.06bp) and (149.82bp,554.21bp)  .. (15);
			\draw [->] (30) ..controls (306.49bp,352.46bp) and (297.57bp,340.78bp)  .. (28);
			\draw [->] (30) -- (29); %..controls (332.08bp,352.24bp) and (334.34bp,342.67bp)  .. (29);
			\draw [->] (30) --(11);%..controls (284.46bp,354.45bp) and (252.09bp,336.1bp)  .. (11);
			\draw [->] (27) ..controls (242.97bp,209.09bp) and (227.16bp,195.03bp)  .. (4);
			\draw [->] (27) ..controls (267.77bp,208.13bp) and (266.62bp,198.97bp)  .. (43);
			\draw [->] (28)--(5);% ..controls (224.78bp,282.89bp) and (185.24bp,263.12bp)  .. (5);
			\draw [->] (28) ..controls (271.0bp,280.13bp) and (271.0bp,270.97bp)  .. (27);
			\draw [->] (32) ..controls (326.0bp,496.13bp) and (326.0bp,486.97bp)  .. (31);
			\draw [->] (32) ..controls (300.8bp,496.8bp) and (287.68bp,483.68bp)  .. (19);
			\draw [->] (25) ..controls (152.2bp,712.8bp) and (165.32bp,699.68bp)  .. (22);
			\draw [->] (25) ..controls (127.0bp,712.13bp) and (127.0bp,702.97bp)  .. (24);
			\draw [->] (7) --(11);%controls (234.49bp,352.46bp) and (225.57bp,340.78bp)  .. (11);
			\draw [->] (4) ..controls (190.0bp,136.13bp) and (190.0bp,126.97bp)  .. (3);
			\draw [->] (31) ..controls (326.0bp,424.13bp) and (326.0bp,414.97bp)  .. (30);
			\draw [->] (31) ..controls (300.8bp,424.8bp) and (287.68bp,411.68bp)  .. (7);
			\draw [->] (9) ..controls (228.8bp,496.8bp) and (215.68bp,483.68bp)  .. (13);
			\draw [->] (9) ..controls (254.0bp,496.13bp) and (254.0bp,486.97bp)  .. (19);
			\draw [->] (34) ..controls (271.0bp,712.13bp) and (271.0bp,702.97bp)  .. (33);
			\draw [->] (15) ..controls (135.2bp,496.8bp) and (148.32bp,483.68bp)  .. (13);
			\draw [->] (15) ..controls (110.0bp,496.13bp) and (110.0bp,486.97bp)  .. (14);
			\draw [->] (6) --(4);%..controls (195.77bp,208.13bp) and (194.62bp,198.97bp)  .. (4);
			\draw [->] (36) ..controls (167.83bp,856.66bp) and (156.74bp,843.99bp)  .. (1);
			\draw [->] (36) ..controls (193.23bp,856.13bp) and (194.38bp,846.97bp)  .. (40);
			\draw [->] (37) ..controls (55.0bp,1000.1bp) and (55.0bp,990.97bp)  .. (39);
			\draw [->] (37) ..controls (34.156bp,999.64bp) and (24.765bp,985.71bp)  .. (19.0bp,972.0bp) .. controls (3.1542bp,934.31bp) and (0.0bp,922.88bp)  .. (0.0bp,882.0bp) .. controls (0.0bp,882.0bp) and (0.0bp,882.0bp)  .. (0.0bp,306.0bp) .. controls (0.0bp,228.53bp) and (56.874bp,150.79bp)  .. (42);
			\draw [->] (10) ..controls (116.08bp,352.24bp) and (118.34bp,342.67bp)  .. (8);
			\draw [->] (10) ..controls (90.488bp,352.46bp) and (81.565bp,340.78bp)  .. (21);
			\draw [->] (8)--(5);% ..controls (127.0bp,280.13bp) and (127.0bp,270.97bp)  .. (5);
			\draw [->] (41) ..controls (127.0bp,1000.1bp) and (127.0bp,990.97bp)  .. (26);
			\draw [->] (41) ..controls (101.8bp,1000.8bp) and (88.685bp,987.68bp)  .. (39);
			\draw [->] (35) ..controls (55.0bp,527.29bp) and (55.0bp,392.8bp)  .. (21);
			\draw [->] (29) ..controls (317.8bp,280.8bp) and (304.68bp,267.68bp)  .. (27);
			\draw [->] (29) ..controls (296.78bp,282.89bp) and (257.24bp,263.12bp)  .. (6);
			\draw [->] (38)--(1);% ..controls (121.23bp,856.13bp) and (122.38bp,846.97bp)  .. (1);
			\draw [->] (38)--(2);% ..controls (95.827bp,856.66bp) and (84.739bp,843.99bp)  .. (2);
			\draw [->] (45) ..controls (101.8bp,1072.8bp) and (88.685bp,1059.7bp)  .. (37);
			\draw [->] (45) ..controls (127.0bp,1072.1bp) and (127.0bp,1063.0bp)  .. (41);
			\draw [->] (45) ..controls (146.39bp,1070.9bp) and (155.74bp,1057.0bp)  .. (163.0bp,1044.0bp) .. controls (203.54bp,971.55bp) and (242.05bp,881.2bp)  .. (44);
			\draw [->] (18) ..controls (120.92bp,568.24bp) and (118.66bp,558.67bp)  .. (15);
			\draw [->] (18) ..controls (146.51bp,568.46bp) and (155.43bp,556.78bp)  .. (17);
			\draw [->] (39) ..controls (55.0bp,911.2bp) and (55.0bp,867.25bp)  .. (2);
			\draw [->] (22) ..controls (199.0bp,640.13bp) and (199.0bp,630.97bp)  .. (16);
			\draw [->] (22) ..controls (173.8bp,640.8bp) and (160.68bp,627.68bp)  .. (18);
			\draw [->] (22) ..controls (224.2bp,640.8bp) and (237.32bp,627.68bp)  .. (20);
			\draw [->] (17) ..controls (182.0bp,496.13bp) and (182.0bp,486.97bp)  .. (13);
			\draw [->] (43) ..controls (235.63bp,148.7bp) and (230.66bp,146.25bp)  .. (226.0bp,144.0bp) .. controls (197.52bp,130.25bp) and (164.83bp,115.09bp)  .. (42);
			\draw [->] (43) ..controls (236.8bp,136.8bp) and (223.68bp,123.68bp)  .. (3);
			\draw [->] (20) ..controls (264.92bp,568.24bp) and (262.66bp,558.67bp)  .. (9);
			\draw [->] (20) ..controls (240.17bp,569.06bp) and (221.82bp,554.21bp)  .. (17);
			\draw [->] (42) ..controls (110.0bp,64.131bp) and (110.0bp,54.974bp)  .. (0);
			\draw [->] (3) ..controls (162.32bp,65.085bp) and (146.7bp,51.029bp)  .. (0);
			\draw [->] (44) ..controls (292.7bp,783.97bp) and (302.12bp,770.05bp)  .. (307.0bp,756.0bp) .. controls (326.0bp,685.22bp) and (326.0bp,595.42bp)  .. (32);
			\draw [->] (44) ..controls (271.0bp,784.13bp) and (271.0bp,774.97bp)  .. (34);
			\draw [->] (13) ..controls (207.2bp,424.8bp) and (220.32bp,411.68bp)  .. (7);
			\draw [->] (13) ..controls (182.0bp,424.13bp) and (182.0bp,414.97bp)  .. (12);
			\draw [->] (11) ..controls (173.8bp,280.8bp) and (170.68bp,267.68bp)  .. (5);
			\draw [->] (11) --(6);%..controls (199.0bp,280.13bp) and (199.0bp,270.97bp)  .. (6);
			\draw [->] (33) ..controls (271.0bp,640.13bp) and (271.0bp,630.97bp)  .. (20);
			\draw [->] (19) ..controls (254.0bp,424.13bp) and (254.0bp,414.97bp)  .. (7);
			\draw [->] (24) ..controls (101.8bp,640.8bp) and (88.685bp,627.68bp)  .. (35);
			\draw [->] (24) ..controls (127.0bp,640.13bp) and (127.0bp,630.97bp)  .. (18);
			\draw [->] (14) ..controls (110.0bp,424.13bp) and (110.0bp,414.97bp)  .. (10);
			\draw [->] (14) ..controls (135.2bp,424.8bp) and (148.32bp,411.68bp)  .. (12);
			\draw [->] (12) ..controls (162.49bp,352.46bp) and (153.57bp,340.78bp)  .. (8);
			\draw [->] (12) --(11);%..controls (188.08bp,352.24bp) and (190.34bp,342.67bp)  .. (11);
			\draw [->] (40) ..controls (199.0bp,784.13bp) and (199.0bp,774.97bp)  .. (23);
			\draw [->] (40) ..controls (224.2bp,784.8bp) and (237.32bp,771.68bp)  .. (34);
			\draw [->] (21) ..controls (50.661bp,247.51bp) and (47.371bp,148.61bp)  .. (74.0bp,72.0bp) .. controls (77.662bp,61.464bp) and (83.876bp,51.001bp)  .. (0);
		\end{tikzpicture}
		\caption{The Hasse diagram of the $3$-torsion classes in the $3$-cluster tilting subcategory $\mc_3^3$ of the higher Auslander Algebra $A_{3}^{3}$.}
		\label{fig:Hasse}
	\end{figure}
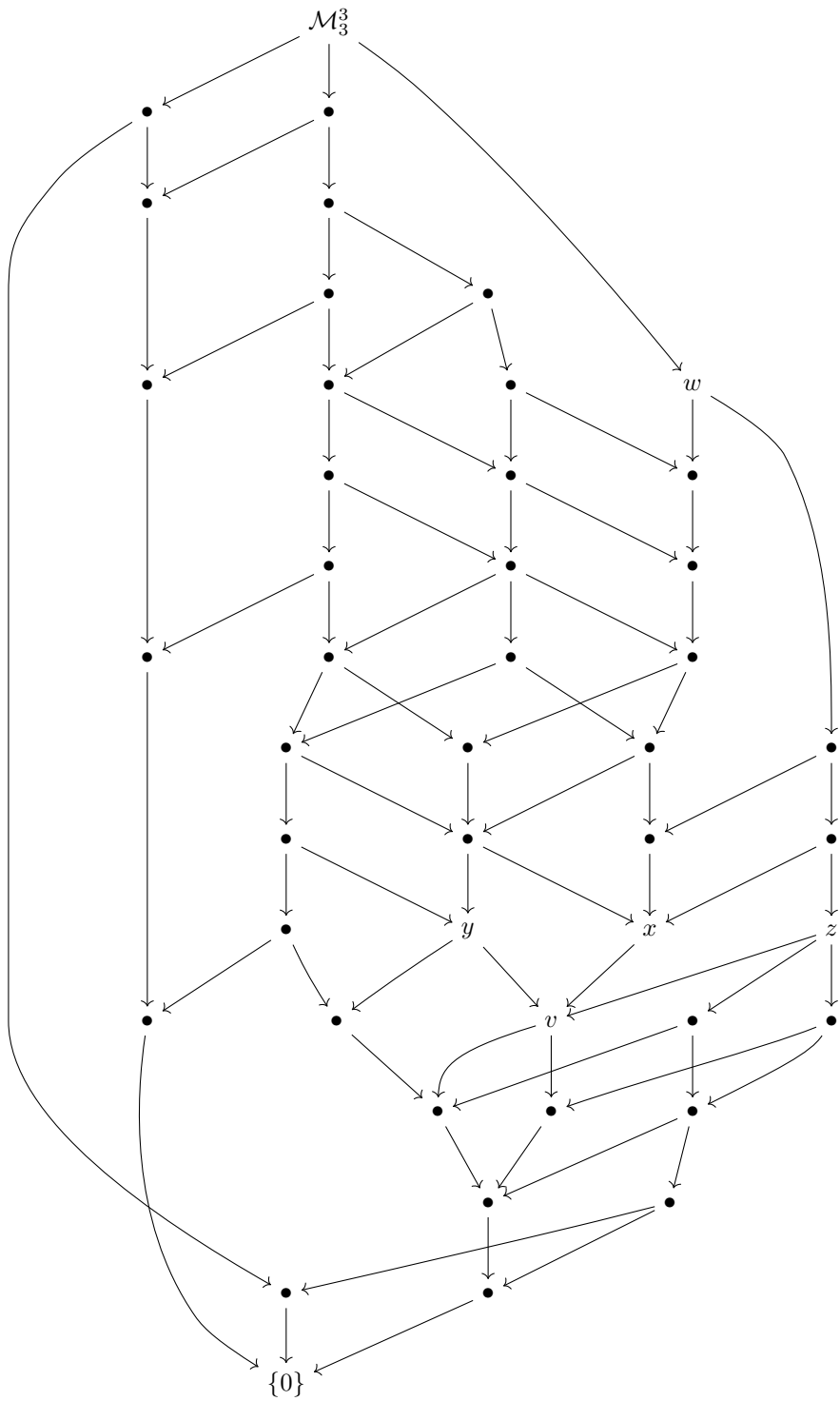 
	
	\section{Higher Nakayama algebras}\label{sec:Higher NA}
	Higher Nakayama algebras were introduced in \cite{Higher Nakayama} as a higher-dimensional generalisation of classical Nakayama algebras. In this section we extend the combinatorial description of higher torsion classes from \cref{thm:higherAuslanderalg} to the setup of higher Nakayama algebras. We first consider higher Nakayama algebras of type $\mathbb{A}$ in \cref{subsec:higher NA type A}, before moving on to type $\mathbb{A}_\infty^\infty$ in \cref{subsec: higher NA type A-infinity}.
	
	\subsection{Higher Nakayama algebras of type \texorpdfstring{$\mathbb{A}$}{A}}\label{subsec:higher NA type A}
	
	We start by giving a brief introduction to the construction of higher Nakayama algebras of type $\mathbb{A}$. Let $n$ and $d$ be positive integers, and recall the definitions of $\os_{n}^{d}$, $A_n^d$, $\mc_n^d$ and $M_x$ from \cref{Background on higher Auslander algebras}.

	A \textit{(connected) Kupisch series of type }$\mathbb{A}_n$ is a tuple $\underline{\ell}=(\ell_0,\ell_1,\dots, \ell_{n-1})$ of positive integers satisfying 
	\[
	\ell_0=1 \quad \text{and} \quad 2\leq \ell_i\leq \ell_{i-1}+1  \quad \text{for} \quad i=1,\ldots,n-1. 
	\]
	Given such a Kupisch series $\underline{\ell}$, consider the subset
	\[
	\os_{\underline{\ell}}^{d+1}\colonequals\{y\in \os_{n}^{d+1}\mid \ell\ell(y)\leq \ell_{y_d}\}
	\]
	where $\ell\ell(y)=y_d-y_0+1$. Note that $\ell\ell(y)$ is equal to the Loewy length of the module $M_y$ in $\operatorname{mod}A_n^d$, see \cite[Lemma 2.9]{Higher Nakayama}. For $d\geq 2$, the $d$-\textit{th Nakayama algebra with Kupisch series $\underline{\ell}$} is the idempotent quotient
	\[
	A_{\underline{\ell}}^d \colonequals A_n^d/(\os_{n}^{d}\setminus\os_{\underline{\ell}}^{d}).
	\]
	In other words, if we let $e_{\underline{\ell}}$ be the idempotent consisting of the sum of the vertices in $\os_{n}^{d}\setminus\os_{\underline{\ell}}^{d}$, then $A_{\underline{\ell}}^d$ is isomorphic to $A_n^d/A_n^d e_{\underline{\ell}}A_n^d$. It follows from \cite[Proposition 2.24]{Higher Nakayama} that the subcategory $\mc_{\underline{\ell}}^d \colonequals \mc_{n}^d\cap \operatorname{mod}A_{\underline{\ell}}^d$ is $d$-cluster tilting in $\operatorname{mod}A_{\underline{\ell}}^d$, and that $\mc_{\underline{\ell}}^d=\operatorname{add}(M_{\underline{\ell}}^d)$ for $M_{\underline{\ell}}^d=\bigoplus_{x\in \os_{\underline{\ell}}^{d+1}}M_x$.
	
	Note that the isomorphism classes of indecomposable modules in $\mc_{\underline{\ell}}^d$ are in bijection with elements of $\os_{\underline{\ell}}^{d+1}$. Using our results, we can characterise the subsets of $\os_{\underline{\ell}}^{d+1}$ which correspond to higher torsion classes in $\mc_{\underline{\ell}}^d$. We use the notation
	\begin{align*}
		\uc_I\colonequals \add \{ M_{y} \in \mc_{\underline{\ell}}^d \mid y \in I\}
	\end{align*}
	for the subcategory of $\mc_{\underline{\ell}}^d$ associated to a subset $I \subseteq \os_{\underline{\ell}}^{d+1}$.
	
	\begin{theorem}\label{thm:higherNakayamaAlg}
		Let $\underline{\ell}$ be a Kupisch series of type $\mathbb{A}_n$ and consider a subset $I \subseteq \os_{\underline{\ell}}^{d+1}$. The subcategory $\uc_I$ is a $d$-torsion class in $\mc_{\underline{\ell}}^d$ if and only if the following hold for any elements $x,z\in \os_{\underline{\ell}}^{d+1}$:
		\begin{enumerate}
			\item \label{thm:higherNakayamaAlg:1} If $x\leq z$ and $x_d=z_d$, then $x\in I$ implies $z\in I$.
			\item \label{thm:higherNakayamaAlg:2} If $x\rightsquigarrow \tau_d(z)$ and $x,z\in I$, then any $y\in \os_{\underline{\ell}}^{d+1}$ with $y_i \in \{x_i,z_i\}$ for each $i$ must be in $I$. 
		\end{enumerate}
	\end{theorem}
	
	\begin{proof}
		By \cref{cor: dqOnSumOfIndec}, condition (\ref{thm:higherNakayamaAlg:1}) is equivalent to $\uc_I$ being closed under $d$-quotients in $\mc^d_{n}$. For objects in $\mc_{\underline{\ell}}^d$, the minimal $d$-quotients in $\mc_{\underline{\ell}}^d$ are the same as the minimal $d$-quotients in $\mc^d_{n}$ by \cref{prop:IntersectionMapOfPoset}\eqref{prop:IntersectionMapOfPoset:3} and \cref{rem: d-quotients inherited}. Hence, condition (\ref{thm:higherNakayamaAlg:1}) is also equivalent to $\uc_I$ being closed under $d$-quotients in $\mc_{\underline{\ell}}^d$. 
		
		Note next that $d$-extensions in $\mc_{\underline{\ell}}^d$ coincide with $d$-extensions in $\mc_n^d$ with all terms in $\mc_{\underline{\ell}}^d$. By \cref{prop: ext descrip} and \cref{rem: min d-ext}, condition (\ref{thm:higherNakayamaAlg:2}) hence implies that $\uc_I$ is closed under \mbox{$d$-extensions} by indecomposables in $\mc_{\underline{\ell}}^d$. Assuming both (\ref{thm:higherNakayamaAlg:1}) and (\ref{thm:higherNakayamaAlg:2}) thus yields that $\uc_I$ is closed under all $d$-extensions in $\mc_{\underline{\ell}}^d$ by \cref{prop: extension closure indec}, so $\uc_I$ is a $d$-torsion class in $\mc_{\underline{\ell}}^d$ by \cref{cor: characterisation}.
		
		It remains to show that if $\uc_I$ is a $d$-torsion class in $\mc_{\underline{\ell}}^d$, then condition (\ref{thm:higherNakayamaAlg:2}) is satisfied. We will use that we already know condition (\ref{thm:higherNakayamaAlg:1}) holds. Consider $x, z\in I$ with $x\rightsquigarrow \tau_d(z)$, and suppose that $y\in \os_{\underline{\ell}}^{d+1}$ satisfies  $y_i \in \{x_i,z_i\}$ for each $i$. We need to show that $y \in I$.
		
		As $x\rightsquigarrow \tau_d(z)$, there is a minimal $d$-extension
		\begin{align}\label{d-extension HNA}
			0\rightarrow M_x \rightarrow E_1\rightarrow \cdots \rightarrow E_d \rightarrow M_{z} \rightarrow 0
		\end{align}
		in $\mc_n^d$ with $M_y$ as a direct summand in one of the middle terms by \cref{prop: ext descrip} and \cref{rem: min d-ext}. Let us first assume $z_d-x_0+1\leq \ell_{z_d}$. Using that $x, z\in \os_{\underline{\ell}}^{d+1}$ and $x_0 \leq z_0-1$, we see that all the terms in (\ref{d-extension HNA}) are in $\mc_{\underline{\ell}}^d$ in this case. Hence, we must have $M_y\in \uc_I$ and $y\in I$ as $\uc_I$ is closed under $d$-extensions in $\mc_{\underline{\ell}}^d$.
		
		Consider now the case $z_d-x_0+1> \ell_{z_d}$. Note that we have $x\leq y$. If $y_d=x_d$, condition (\ref{thm:higherNakayamaAlg:1}) hence yields that $y\in I$, so we can assume $y_d=z_d$. Now $y_0=x_0$ would contradict the assumption $z_d-x_0+1> \ell_{z_d}$ as $y\in \os_{\underline{\ell}}^{d+1}$, so we must have $y_0=z_0$.
		
		If $y=z$, we have $y\in I$, so assume that $y_i=x_i$ for some $1\leq i\leq d-1$. This ensures that $k \colonequals \min\{i\mid z_{i-1}\leq x_i\}$ exists. Note that for $j<k$, we have $z_{j-1}-1\geq x_j$. Combining this with the inequality 
		$x_j \geq z_{j-1}-1$
		coming from $x\rightsquigarrow \tau_d(z)$, we get that $x_j=z_{j-1}-1$ for $j>0$.
		Let $$w=(z_0,z_1,\ldots, z_{k-1}, x_k,\ldots, x_{d-1}, z_d)=(x_1+1,x_2+1,\ldots, x_{k-1}+1,z_{k-1}, x_k,\ldots, x_{d-1}, z_d),$$ and observe that $w \in \os_{\underline{\ell}}^{d+1}$ with $w\leq y$ and $w_d=y_d$ since $y$ is of the form $$y=(z_0, z_1, \dots, z_{k-1}, y_k, \dots, y_{d-1}, z_d)$$ under our current assumptions. 
		Consequently, it suffices to show that $w\in I$, as this implies $y \in I$ by condition (\ref{thm:higherNakayamaAlg:1}). 
		
		To this end, define
		\begin{align*}
			\overline{x}&=(z_0,z_1,\ldots, z_{k-1}, x_k, \ldots, x_d) \\
			\overline{z}&=(z_0+1,z_1+1,\ldots, z_{k-1}+1, z_k, \ldots, z_d).
		\end{align*}
		We see that $\overline{x},\overline{z} \in \os_{\underline{\ell}}^{d+1}$.
		Observe moreover that $x\leq \overline{x}$ with $\overline{x}_d=x_d$ and $z\leq \overline{z}$ with $\overline{z}_d=z_d$, so $\overline{x},\overline{z}\in I$ by condition (\ref{thm:higherNakayamaAlg:1}). Furthermore, we have $\overline{x}\rightsquigarrow \tau_d(\overline{z})$ and $w_i\in \{\overline{x}_i,\overline{z}_i\}$ for all $i$. Finally, notice that $\overline{z}_d-\overline{x}_0+1=z_d-z_0+1\leq \ell_{\overline{z}_d}$. It follows that $w\in I$ by the same argument as earlier in this proof, and we can conclude that $y \in I$.
	\end{proof}
	
	The following example illustrates the use of \cref{thm:higherNakayamaAlg}.
	
	\begin{example}\label{ex:higher nakayama}
		Fix $n=4$, $d=2$ and $\underline{\ell} = (1,2,2,3)$. The higher Nakayama algebra $A_{\underline{\ell}}^{2}$ has a $2$-cluster tilting subcategory $\mc_{\underline{\ell}}^{2}$, whose associated quiver can be found in \cref{fig:Higher Nakayama}. Note that we identify the indecomposable objects of $\mc_{\underline{\ell}}^{2}$ with $\os_{\underline{\ell}}^{3}$, and that we use a shortened notation for the sake of simplicity. The following subcategories are examples of $2$-torsion classes in $\mc_{\underline{\ell}}^{2}$:
		\begin{itemize}
			\item $\add\{0\}$
			\item $\mc_{\underline{\ell}}^{2}$
			\item $\add \{000, 133, 222, 223, 233, 333\}$
			\item $\add \{112,113,122,123,222,133,223,233,333\}$
		\end{itemize}
		
		\begin{figure}[htb]
			\centering
			\begin{tikzpicture}[xscale=1]
				\node (000) at (-1,0) {000};
				\node (111) at (3,0) {111};
				\node (222) at (7,0) {222};
				\node (333) at (11,0) {333};
				
				\node (001) at (0,1) {001};
				\node (011) at (2,1) {011};
				\node (112) at (4,1) {112};
				\node (122) at (6,1) {122};	
				\node (223) at (8,1) {223};
				\node (233) at (10,1) {233};
				
				\node (113) at (5,2) {113};
				\node (123) at (7,2) {123};	
				\node (133) at (9,2) {133};
				
				\draw[->] (000) -- (001);
				\draw[->] (001) -- (011);
				\draw[->] (011) -- (111);
				\draw[->] (111) -- (112);
				\draw[->] (112) -- (122);
				\draw[->] (122) -- (222);
				\draw[->] (112) -- (113);
				\draw[->] (113) -- (123);
				\draw[->] (123) -- (223);
				\draw[->] (122) -- (123);
				\draw[->] (123) -- (133);
				\draw[->] (133) -- (233);
				\draw[->] (222) -- (223);
				\draw[->] (223) -- (233);
				\draw[->] (233) -- (333);
				
				\draw[->, dashed] (111)--(000);
				\draw[->, dashed] (222)--(111);
				\draw[->, dashed] (333)--(222);
				
				\draw[->, dashed] (223) to[bend left=23] (112);
				\draw[->, dashed] (233) to[bend left=23] (122);
				
			\end{tikzpicture}
			\caption{
				The quiver of the $2$-cluster tilting subcategory $\mc_{\underline{\ell}}^{2}$ in \cref{ex:higher nakayama}}.
			\label{fig:Higher Nakayama}
		\end{figure}
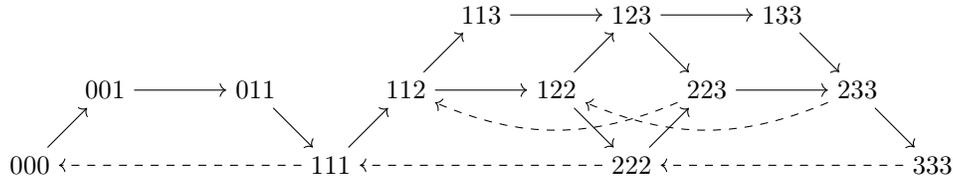
		
	\end{example}

	\begin{remark}\label{rem:results-higher-Nakayama}
		With the results in this section, we can extend \Cref{alg:make torsion class,alg: find torsion classes}  to the higher Nakayama setting with a few changes:
		\begin{itemize}
			\item The input needs to include the Kupisch series $\underline{\ell}$.
			\item Elements must be chosen from $\os_{\underline{\ell}}^{d+1}$, rather than $\os_n^{d+1}$.
		\end{itemize}
		We have implemented the extended algorithm in the Google Colab Notebook associated to this paper.\footnote{\url{https://colab.research.google.com/drive/172Q-UZHvdPOhngGkl1T_xdYLzntg31dY}}
	\end{remark}
	
	\subsection{Higher Nakayama algebras of type \texorpdfstring{$\mathbb{A}_\infty^\infty$}{A-infinity}}
	\label{subsec: higher NA type A-infinity}
	
	The goal of this subsection is to further extend the combinatorial description of $d$-torsion classes to higher Nakayama algebras of type $\mathbb{A}_\infty^\infty$ as introduced in \cite{Higher Nakayama}. We start by giving a brief introduction to this class of algebras.
	
	In contrast to earlier in this paper, we now need to consider quivers with infinitely many vertices. Hence, associated to a quiver with relations is a category $\cc$ rather than an algebra. The objects of $\cc$ are the vertices of the quiver, and a basis of the morphisms spaces are given by the arrows modulo the given relations. A \textit{right module over} $\cc$ is a $\kk$-linear functor $M\colon \cc^{\operatorname{op}}\to \operatorname{Mod}\kk$. The module $M$ is called \textit{finite-dimensional} if the sum $\bigoplus_{x\in \cc}M(x)$ is finite-dimensional. We let $\operatorname{fd}\cc$ denote the category of finite-dimensional right modules over $\cc$. For more details, see \cite[Sections 1.1 and 1.2]{Higher Nakayama}.
	
	Let $\os^d$ denote the set of non-decreasing $d$-tuples $x=(x_0,\dots , x_{d-1})$ over $\mathbb{Z}$. Consider $\mathbb{Z}^d$ as a $\kk$-linear category given by the $\kk$-linearisation of the poset $\mathbb{Z}^d$ endowed with the product order. The \textit{mesh category of type }$\mathbb{Z}\mathbb{A}_\infty^{d-1}$ is defined to be the additive quotient
	\[
	A_\infty^{d} \colonequals \mathbb{Z}^d/(\mathbb{Z}^d\setminus \os^d).
	\] 
	Note that $A_\infty^{d}$ can be represented by the opposite of an infinite quiver $Q^{d}$ with relations. The vertices of $Q^d$ are elements in $\os^d$, and there is an arrow from vertex $x$ to vertex $y$ if $y_i = x_i + 1$ for exactly one $i= 0,\ldots,  d-1$ and $y_j = x_j$ for $j\neq i$. The relations of $A_\infty^{d}$ are given by a certain admissible ideal making squares commutative and sending certain compositions of two arrows to zero, see \cite[Section 3.1]{Higher Nakayama} for more details. 
	
	Following \cite[Appendix B]{Higher Nakayama}, we say that a \textit{Kupisch series of type} $\mathbb{A}_\infty^\infty$ is an infinite tuple $\underline{\ell} =(\dots, \ell_{-1}, \ell_0, \ell_1, \dots )$ where $\ell_i$ is either a non-negative integer or equal to $\infty$, and where the inequality
	\[
	\ell_i\leq \ell_{i-1}+1
	\]
	holds for all $i\in \mathbb{Z}$. We write $\operatorname{KS}(\mathbb{A}_\infty^\infty)$ for the set of Kupisch series of type $\mathbb{A}_\infty^\infty$. Given such a Kupisch series $\underline{\ell}$, define the subset
	\[
	\os_{\underline{\ell}}^{d+1}\colonequals\{y\in \os^{d+1}\mid \ell\ell(y)\leq \ell_{y_{d}}\},
	\]
	where $\ell\ell(y)=y_d-y_0+1$ as before. For $d\geq 2$, the category $A_{\underline{\ell}}^d$ is defined as the idempotent quotient
	\[
	A_{\underline{\ell}}^d \colonequals A_{\infty}^{d}/(\os^d\setminus\os_{\underline{\ell}}^{d}).
	\]
	
	Note that we have an inclusion $\operatorname{fd}A_{\underline{\ell}}^d\to \operatorname{fd}A_\infty^d$ between the categories of finite-dimensional modules. Associated to the Kupisch series $\underline{\ell}$, we also have the subcategory 
	\[
	\mc_{\underline{\ell}}^d\colonequals \add\{M_x \in \operatorname{fd}A_\infty^d \mid x\in \os_{\underline{\ell}}^{d+1}\}
	\]
	of $\operatorname{fd}A_\infty^d$. Here, we use the notation $M_x$ for the indecomposable $A_\infty^d$-module with support in all vertices $y\in \os^d$ satisfying $x_0\leq y_0\leq x_1\leq \cdots \leq x_{d-1}\leq y_{d-1}\leq x_d$. This extends the analogous definition in \cref{Background on higher Auslander algebras}. By \cite[Appendix B]{Higher Nakayama}, the subcategory $\mc_{\underline{\ell}}^d$ is $d$-cluster tilting in $\operatorname{fd}A_{\underline{\ell}}^d$. For certain choices of $\underline{\ell}$, this gives an example of a $d$-cluster tilting subcategory of an abelian category with no non-zero projective or injective objects, e.g.\ if $\ell_i=\infty$ for all $i\in \mathbb{Z}$, see \cite[Proposition 3.6]{Higher Nakayama}.   
	
	Now consider the partial order on the set $\operatorname{KS}(\mathbb{A}_\infty^\infty)$ given by the product order. If $\underline{\ell}\leq \underline{\ell'}$, then we have a natural functor $A_{\underline{\ell}'}^d\to A_{\underline{\ell}}^d$ of categories. This gives an inclusion $\operatorname{fd}A_{\underline{\ell}}^d\to \operatorname{fd}A_{\underline{\ell}'}^d$ such that the equality
	\[
	\mc_{\underline{\ell}}^d=\mc_{\underline{\ell}'}^d\cap \operatorname{fd}A_{\underline{\ell}}^d
	\]
	holds. Similarly, if $\underline{\ell}^1\leq \underline{\ell}^2\leq \cdots$ is an increasing sequence of Kupisch series which converges to $\underline{\ell}\in \operatorname{KS}(\mathbb{A}_\infty^\infty)$ (in the natural way), then 
	\[
	\operatorname{fd}A_{\underline{\ell}}^d=\bigcup_{i\geq 1}\operatorname{fd}A_{\underline{\ell}^i}^d \quad \text{and} \quad \mc_{\underline{\ell}}^d=\bigcup_{i\geq 1}\mc_{\underline{\ell}^i}^d.
	\]
	A Kupisch series $\underline{\ell}\in \operatorname{KS}(\mathbb{A}_\infty^\infty)$ is called \textit{finite} if $\ell_j=0$ for all but finitely many $j\in \mathbb{Z}$. Note that for any $\underline{\ell}\in \operatorname{KS}(\mathbb{A}_\infty^\infty)$, we can find a sequence $\underline{\ell}^1\leq \underline{\ell}^2\leq \cdots$ of finite Kupisch series in $\operatorname{KS}(\mathbb{A}_\infty^\infty)$ which converges to $\underline{\ell}$. We use this to give a characterisation of the $d$-torsion classes in $\mc_{\underline{\ell}}^d$ in \cref{thm:higherNakayamaAlgInfinite} below.
	
	\begin{remark}\label{Remark: Finite Kupisch series}
		Let $\underline{\ell}$ be a Kupisch series of type $\mathbb{A}_n$. Then $\underline{\ell}$ can be identified with a finite Kupisch series $\underline{\ell}'$ of type $\mathbb{A}_\infty^\infty$ which is non-zero only in positions $0,\ldots, n-1$.  In this case there is a bijection between the sets $\os_{\underline{\ell}}^{d}$ and $\os_{\underline{\ell}'}^{d}$, so the algebra $A^d_{\underline{\ell}}$ and the category $A^d_{\underline{\ell}'}$ can be naturally identified. Note that up to isomorphism, the set $\os_{\underline{\ell}'}^{d}$ and the category $A^d_{\underline{\ell}'}$ remain unchanged when shifting $\underline{\ell}'$ some number of steps to the left or right.
		
		In general, if $\underline{\ell}$ is a finite Kupisch series of type $\mathbb{A}^\infty_\infty$, then $\underline{\ell}$ is obtained by gluing together shifts of Kupisch series of type $\mathbb{A}$. Hence, the set $\os^d_{\underline{\ell}}$ is in bijection with a disjoint union $\bigcup_{j=1}^m\os^d_{\underline{\ell}_j}$ where $\underline{\ell}_j$ is a Kupisch series of type $\mathbb{A}_{n_j}$ for some integer $n_j\geq 1$. The associated category $A^d_{\underline{\ell}}$ can thus be identified with a finite product \[A^d_{\underline{\ell}_1}\times \cdots \times A^d_{\underline{\ell}_m}\] where $A^d_{\underline{\ell}_j}$ is a higher Nakayama algebras of type $\mathbb{A}_{n_j}$. With this identification, the category $\operatorname{fd}A^d_{\underline{\ell}}$ 
		is equivalent to the product $\operatorname{mod}A^d_{\underline{\ell}_1}\times \cdots \times\operatorname{mod}A^d_{\underline{\ell}_m}$, and the $d$-cluster tilting subcategory $\mc^d_{\underline{\ell}}$ is equivalent to the product $\mc^d_{\underline{\ell}_1}\times \cdots \times \mc^d_{\underline{\ell}_m}$. Since any Kupisch series $\underline{\ell}$ of type $\mathbb{A}_\infty^\infty$ can be represented by a converging sequence of finite Kupisch series, it follows that the associated $d$-cluster tilting subcategory $\mc_{\underline{\ell}}^d$ is the union of finite products of $d$-cluster tilting subcategories of higher Nakayama algebras of type $\mathbb{A}$.
	\end{remark}
	
	As before, we use the notation
	\begin{align*}
		\uc_I\colonequals \add \{ M_{y} \in \mc_{\underline{\ell}}^d \mid y \in I\}
	\end{align*}
	for the subcategory of $\mc_{\underline{\ell}}^d$ associated to a subset $I \subseteq \os_{\underline{\ell}}^{d+1}$.
	
	\begin{theorem}\label{thm:higherNakayamaAlgInfinite}
		Let $\underline{\ell}$ be a Kupisch series of type $\mathbb{A}_\infty^\infty$ and consider a subset $I \subseteq \os_{\underline{\ell}}^{d+1}$.
		The subcategory $\uc_I$ is a $d$-torsion class in $\mc_{\underline{\ell}}^d$ if and only if the following hold for any elements $x,z\in \os_{\underline{\ell}}^{d+1}$:
		\begin{enumerate}
			\item \label{thm:HigherNakayamainf:1} If $x\leq z$ and $x_d=z_d$, then $x\in I$ implies $z\in I$.
			\item \label{thm:HigherNakayamainf:2} If $x\rightsquigarrow \tau_d(z)$ and $x,z\in I$, then any $y\in \os_{\underline{\ell}}^{d+1}$ with $y_i \in \{x_i,z_i\}$ for each $i$ must be in $I$. 
		\end{enumerate}
	\end{theorem}
	
	\begin{proof}
		Choose an increasing sequence $\underline{\ell}^1\leq \underline{\ell}^2\leq \cdots$ of finite Kupisch series which converges to $\underline{\ell}$. Then we have that 
		\[
		\uc_I\cap \mc^d_{\underline{\ell}^i}=\uc_{I\cap \os_{\underline{\ell}^i}^{d+1}}
		\]
		for each $i \geq 1$. Fix $i\geq 1$, and set $J^i=I\cap \os_{\underline{\ell}^i}^{d+1}$. We have an equivalence between $\mc^d_{\underline{\ell}^i}$ and a finite product $\mc^d_{\underline{\ell}^i_1}\times \cdots \times\mc^d_{\underline{\ell}^i_m}$ as in \cref{Remark: Finite Kupisch series}. The subcategory $\uc_{J^i}$ of $\mc^d_{\underline{\ell}^i}$ is hence equivalent to a product $\uc_1\times \cdots \times \uc_m$, where $\uc_k$ is a subcategory of $\mc^d_{\underline{\ell}^i_k}$ for each $k$. Note that $\os_{\underline{\ell}^i}^{d+1}$ is in bijection with the disjoint union $\bigcup_{k=1}^m\os^{d+1}_{\underline{\ell}^i_k}$ as in \cref{Remark: Finite Kupisch series}. We let $J^i_k$ denote the intersection of $\os^{d+1}_{\underline{\ell}^i_k}$ with the image of $J^i$ under this bijection. Then we get $\uc_k=\uc_{J^i_k}$.
		
		Note that $I$ satisfies the conditions $(\ref{thm:HigherNakayamainf:1})$ and $(\ref{thm:HigherNakayamainf:2})$ in the statement if and only if the set $J^i$ satisfies the analogous conditions for each $i\geq 1$. Furthermore, this holds if and only if each $J^i_k$ satisfies the conditions of \cref{thm:higherNakayamaAlg}. By \cref{thm:higherNakayamaAlg}, this is again equivalent to $\uc_{J^i_k}$ being a $d$-torsion class in $\mc^d_{\underline{\ell}^i_k}$ for all $i$ and $k$, i.e.\ that $\uc_{J^i} \simeq
		\uc_{J^i_1}\times \cdots \times \uc_{J^i_m}$ is a $d$-torsion class in $\mc_{\underline{\ell}^i}^d$ for all $i\geq 1$. Hence, it suffices to show that $\uc_I$ is a $d$-torsion class in $\mc_{\underline{\ell}}^d$ if and only if $\uc_{J^i}$ is a $d$-torsion class in $\mc_{\underline{\ell}^i}^d$ for all $i\geq 1$. 
		
		If $\uc_I$ is a $d$-torsion class in $\mc_{\underline{\ell}}^d$, then $\uc_{J^i}=\uc_I\cap \mc^d_{\underline{\ell}^i}$ is a $d$-torsion class in $\mc_{\underline{\ell}^i}^d=\mc_{\underline{\ell}}^d\cap \operatorname{fd}A_{\underline{\ell}^i}^d$ for all $i\geq 1$ by \cref{prop:IntersectionMapOfPoset}\eqref{prop:IntersectionMapOfPoset:1}.  
		Conversely, assume $\uc_{J^i}$ is a $d$-torsion class in $\mc_{\underline{\ell}^i}^d$ for all $i\geq 1$. Since we have
		\[
		\uc_{J^1}\subseteq \uc_{J^2}\subseteq \cdots \quad 
		\text{and} \quad \uc_I=\bigcup_{i\geq 1}\uc_{J^i},
		\]
		the subcategory $\uc_I$ must be closed both under $d$-extensions and $d$-quotients in $\mc_{\underline{\ell}}^d=\bigcup_{i\geq 1}\mc_{\underline{\ell}^i}^d$, since $\uc_{J^i}$ is closed under $d$-extensions and $d$-quotients in $\mc_{\underline{\ell}^i}^d$ for all $i \geq 1$.
		By \cref{cor: characterisation}, this shows that $\uc_I$ is a $d$-torsion class.
	\end{proof}

	\section*{Acknowledgements}
	
	This project was started while the authors participated in the Junior Trimester Program ``New Trends in Representation Theory'' at the Hausdorff Research Institute for Mathematics in Bonn. We thank the Institute for their hospitality and inspiring working conditions.
	
	We are grateful to J.\ Asadollahi, P.\ J\o rgensen and S.\ Schroll for allowing HT to share the ideas of \cite{AJST} even before the first draft of that paper was finished. Without their generosity, this project would not have started. We thank S.\ Land for help with coding optimization. We would also like to thank the anonymous referee for their careful reading and helpful comments.

	JA was supported by the Max Planck Institute of Mathematics, and a DNRF Chair from the Danish National Research Foundation (grant no. DNRF156).
	JA, KMJ and SK would like to thank the Centre for Advanced Study in Oslo, where parts of this work was done.
	
	Parts of this work were carried out while JH visited Aarhus University. She would like to thank the university for the hospitality and the project ``Pure Mathematics in Norway'' funded by the Trond Mohn Foundation for supporting the stay.
	
	KMJ was funded by the Norwegian Research Council via the project "Higher homological algebra and tilting theory" (301046). 
	
	YP was supported by the French grant CHARMS (ANR--19--CE40--0017).
	
	HT was supported by the European Union’s Horizon 2020 research and innovation programme under the Marie Sklodowska-Curie grant agreement No 893654 and by the Deutsche Forschungsgemeinschaft (DFG, German Research Foundation) under Germany's Excellence Strategy Programme -- EXC-2047/1 -- 390685813.
	
	%%%%%%%%%%%%%%%%%%%%%%%%%%%%%%%


\begin{thebibliography}{AHLSV}
		
		\bibitem[AIR]{AdachiIyamaReiten} T.~Adachi, O.~Iyama and I.~Reiten, \textit{$\tau$-tilting theory}, Compos. Math. (2014), vol. 150, no. 3, 415--452.
		
		\bibitem[AHLSV]{Mutation and torsion pairs} L.~Angeleri Hügel, R.~Laking, J.~Šťovíček and J.~ Vitória,
		\textit{Mutation and torsion pairs},
		\textsf{arXiv:2201.02147}.
		
		\bibitem[AHMV]{AngeleriMarksVitoria}
		L.~Angeleri Hügel, F.~Marks and J.~Vitória,
		\newblock \textit{Silting modules,}
		\newblock Int. Math. Res. Not. IMRN 2016, no. 4, 1251--1284. 
		
		\bibitem[AJST]{AJST} J.~Asadollahi, P.~J\o rgensen, S.~Schroll and H.~Treffinger, \textit{On higher torsion classes,}  Nagoya Math. J. 248 (2022), 823--848.
		
		\bibitem[AP]{AP} S.~Asai and C.~Pfeiffer, \textit{Wide subcategories and lattices of torsion classes,} Algebr.~Represent.~Theory 25 (2022), no.~6, 1611--1629.
		
		\bibitem[ASS]{bluebook1} I.~Assem, D.~Simson and A. Skowro{\'n}ski, \textit{Elements of the representation theory of associative algebras: Vol. 1. Techniques of representation theory}. London Mathematical Society Student Texts, 65. Cambridge University Press, Cambridge, 2006.
		
		\bibitem[AHJ+]{AHJKPT2}
		J.~August, J.~Haugland, K.~M.~Jacobsen, S.~Kvamme, Y.~Palu and H.~Treffinger. \textit{Functorially finite higher torsion classes and $\tau_d$-tilting theory}, in preparation. 
		
		\bibitem[BCZ]{BCZ} E.~Barnard, A.~Carrol and S.~Zhu, \textit{Minimal inclusions of torsion classes,} Algebr.~Comb.~ (2019), no.~5, 879--901.
		
		\bibitem[BBD]{BeilinsonBernsteinDeligne}
		A. A. Beĭlinson, J. Bernstein and P. Deligne.
		\newblock \textit{Faisceaux pervers.}
		\newblock Analysis and topology on singular spaces, I (Luminy, 1981), 5–171, 
		Astérisque, 100, Soc. Math. France, Paris, 1982. 
		
		\bibitem[BB]{BrennerButler} S.~Brenner and M.~C.~R.~Butler, \textit{Generalizations of the Bernstein-Gelfand-Ponomarev reflection functors}, Representation theory, II (Proc. Second Internat. Conf., Carleton Univ., Ottawa, Ont., 1979), Lecture Notes in Math.,  vol.~832, pages 103–169. Springer, Berlin-New York, 1980.
		
		\bibitem[B]{Bridgeland} T.~Bridgeland,
		\textit{Stability conditions on triangulated categories,}
		Ann.~of Math.~(2) 166 (2007), no.~2, 317-–345. 
		
		\bibitem[DIRRT]{DIRRT} L.~Demonet, O.~Iyama, N.~Reading, I.~Reiten and H.~Thomas, \textit{Lattice theory of torsion classes: Beyond $\tau$-tilting theory}, to appear in Trans. Amer. Math. Soc.
		
		\bibitem[D]{Dickson} S.~E.~Dickson, \textit{A torsion theory for abelian categories}, Trans.~Amer.~Math.~Soc., 121 (1966), 223--235.
		
		\bibitem[DJL]{DyckerhoffJassoLekili}
		T.~Dyckerhoff, G.~Jasso, and Y.~Lekili.
		\textit{The symplectic geometry of higher Auslander algebras: symmetric products of disks.}
		Forum Math. Sigma 9, Paper No. e10, 49 p. (2021).
		
		\bibitem[DJW]{DyckerhoffJassoWalde}
		T.~Dyckerhoff, G.~Jasso, and T.~Walde.
		\textit{Simplicial structures in higher Auslander-Reiten theory.}
		Adv.~Math. 355 (2019), 106762.
		
		\bibitem[EN-I]{Ebrahimi} R.~Ebrahimi and A.~ Nasr-Isfahani, \textit{Higher Auslander’s formula}, Int.~Math.~Res.~Not. IMRN 2022, no.~22, 18186-–18203
		
		\bibitem[GM]{GM} A.~Garver and T.~McConville, \textit{Lattice properties of oriented exchange graphs and torsion classes,} Algebr.~Represent.~Theory 22 (2019), no.~1, 43-–78.
		
		\bibitem[GKO]{GeissKellerOppermann} C.~Geiss, B.~Keller and S.~Oppermann,
		\textit{$n$-angulated categories},
		J. Reine Angew. Math. 675 (2013), 101--120.
		
		\bibitem[G]{Graetzer} G.~ Gr\"atzer, \textit{Lattice Theory: Foundation}, Birkhäuser/Springer Basel AG, Basel, 2011. 
		
		\bibitem[H]{Haugland} J.~Haugland, \textit{The Grothendieck group of an $n$-exangulated category}, Appl. Categ. Structures 29 (2021) no.~3, 431--446.
		
		\bibitem[HJ]{HerschendJorgensen} M.~Herschend and P.~J{\o}rgensen, \textit{Classification of higher wide subcategories for higher Auslander algebras of type A}, J.~Pure Appl.~Algebra 225 (2021), no.~5, 106583.
		
		\bibitem[HJV]{HerschendJorgensenVaso} M.~Herschend, P.~J{\o}rgensen and L.~Vaso,
		\textit{Wide subcategories of {$d$}-cluster tilting subcategories},
		Trans. Amer. Math. Soc. 373 (2020), no.~4, 2281--2309.
		
		\bibitem[HLN1]{HerschendLiuNakaoka} M.~Herschend, Y.~Liu and H.~Nakaoka, \textit{$n$-exangulated categories (I): Definitions and fundamental
			properties}, J. Algebra 570 (2021), 531--586.
		
		\bibitem[HLN2]{HerschendLiuNakaoka2} M.~Herschend, Y.~Liu and H.~Nakaoka, \textit{$n$-exangulated categories (II): {C}onstructions from {$n$}-cluster tilting subcategories},
		J. Algebra 594 (2022), 636--684.
		
		\bibitem[I1]{Iyama2007a} O.~Iyama, \textit{Higher-dimensional Auslander--Reiten theory on maximal orthogonal subcategories}, Adv.~Math. 210 (2007), no.~1, 22--50.
		
		\bibitem[I2]{Iyama2007b} O.~Iyama, \textit{Auslander correspondence}, Adv.~Math. 210 (2007), no.~1, 51--82.
		
		\bibitem[I3]{Iyama2011} O.~Iyama,
		\textit{Cluster tilting for higher Auslander algebras},
		Adv.~Math. 226 (2011), no.~1, 1--61. 
		
		\bibitem[IRTT]{IyamaReitenThomasTodorov} O.~Iyama, I.~Reiten, H.~Thomas and G.~Todorov, \textit{Lattice structure of torsion classes for path algebras}, Bull.~Lond.~Math.~Soc. 47, no.~4 (2015), 639--650.
		
		\bibitem[J1]{Jasso2} G.~Jasso, \textit{Reduction of $\tau$-tilting modules and torsion pairs}, Int.~Math.~Res.~Not. IMRN 2015, no.~16, 7190--7237.
		
		\bibitem[J2]{Jasso} G.~Jasso, \textit{$n$-abelian and $n$-exact categories}, Math.~Z. 283 (2016), no.~3-4, 703--759.
		
		\bibitem[JKM]{JassoKellerMuro}
		G.~Jasso, B.~Keller, and F.~Muro,
		\textit{The Triangulated Auslander--Iyama Correspondence.}
		\textsf{arXiv:2208.14413}
		
		\bibitem[JKPK]{Higher Nakayama} G.~Jasso, J.~K\"{u}lshammer, C.~Psaroudakis and S.~Kvamme, \textit{Higher Nakayama algebras I: Construction}, Adv. Math. 351 (2019), 1139--1200.
		
		\bibitem[JK]{higher survey} G.~Jasso and S.~Kvamme, \textit{An introduction to higher Auslander--Reiten theory}, Bull.~Lond.~Math.~Soc. 51 (2019), no.~1, 1–-24.
		
		\bibitem[J{\o}]{Jorgensen} P.~J{\o}rgensen, \textit{Torsion classes and t-structures in higher homological algebra}, Int.~Math.~Res.~Not. IMRN 2016, no.~13, 3880--3905. 
		
		\bibitem[Kl]{Klapproth} C.~Klapproth, \textit{$n$-extension closed subcategories of $n$-exangulated categories}, \textsf{arXiv:2209.01128v3}.
		
		\bibitem[Kr]{Krause} H.~Krause, 
		\textit{Krull-{S}chmidt categories and projective covers},
		Expo. Math. 33 (2015) no.~4, 535--549.
		
		\bibitem[KS]{KrauseSaorin} H.~Krause and M.~Saor\'{\i}n,
		\textit{On minimal approximations of modules},
		in Trends in the representation theory of finite-dimensional algebras ({S}eattle, {WA}, 1997),
		Contemp. Math. 229 (1998), 227--236.
		
		\bibitem[Kv]{Kvamme} S.~Kvamme, \textit{Axiomatizing subcategories of abelian categories}, J.~Pure Appl.~Algebra 226 (2022), no.~4, 106862. 
		
		\bibitem[OT]{OppermannThomas} S.~Oppermann and H.~Thomas, \textit{Higher-dimensional cluster combinatorics and representation theory},
		J.~Eur.~Math.~Soc.~(JEMS) 14 (2012), no. 6, 1679--1737.
		
		\bibitem[T]{Thomas} H.~Thomas, \textit{An introduction to the lattice of torsion classes}, 
		Bull. Iran. Math. Soc. 47, Suppl. 1, 35-55 (2021).
		
	\end{thebibliography}
\end{document}